\documentclass[a4paper,11pt]{amsart}

\usepackage[left=2cm,top=3cm,bottom=2.5cm,right=2.5cm]{geometry}

\usepackage{amsthm}
\usepackage{amssymb}
\usepackage{amsmath}
\usepackage{mathtools}
\usepackage{pdfpages}
\usepackage{changes}
\usepackage{exscale,cite,color,amsopn}
\usepackage[colorlinks=true,pdftex,unicode=true,linktocpage,bookmarksopen,hypertexnames=true]{hyperref}
\usepackage{aliascnt}
\usepackage{csquotes}
\usepackage{colortbl}
\usepackage{enumitem}
\usepackage{verbatim}
\usepackage{hyperref}
\usepackage{tikz-cd}
\usepackage[capitalise]{cleveref}

\usetikzlibrary{trees}

\renewcommand{\leq}{\leqslant}
\renewcommand{\geq}{\geqslant}
\DeclareMathOperator{\id}{id}

\DeclareMathOperator{\Soc}{Soc}
\DeclareMathOperator{\Sym}{Sym}
\DeclareMathOperator{\Sq}{Sq}
\DeclareMathOperator{\ret}{ret}

\DeclareMathOperator{\pr}{pr}
\DeclareMathOperator{\mpl}{mpl}

\DeclareMathOperator{\sca}{sc}

\newcommand{\Z}{\mathbb{Z}}
\newcommand{\F}{\mathbb{F}}

\newcommand{\C}{\mathbb{C}}

\newcommand{\Aut}{\operatorname{Aut}}

\newcommand{\Hol}{\operatorname{Hol}}
\newcommand{\End}{\mathrm{End}}
\newcommand{\Ind}{\mathrm{Ind}}

\newcommand{\Orb}{\mathcal{O}}

\newcommand{\G}{\mathcal{G}}

\newcommand{\genrel}[2]{\left\langle \ #1 \ \vline \ #2 \ \right\rangle}

\makeatletter
\numberwithin{equation}{section}
\numberwithin{figure}{section}
\numberwithin{table}{section}
\newtheorem{thm}{Theorem}[section]
\newtheorem*{thm*}{Theorem}
\newtheorem{lem}[thm]{Lemma}
\newtheorem{cor}[thm]{Corollary}
\newtheorem{pro}[thm]{Proposition}

\theoremstyle{definition}
\newtheorem{defn}[thm]{Definition}
\newtheorem*{defn*}{Definition}

\newtheorem*{convention*}{Convention}
\newtheorem{rem}[thm]{Remark}

\makeatother

\title{Coprime extensions of indecomposable solutions to the Yang--Baxter equation}
\author{Carsten Dietzel}
\date{\today}

\address[Carsten Dietzel]{Laboratoire de Mathématiques Nicolas Oresme, Université de Caen Normandie, Campus 2, 14000 Caen, France}
\email{carsten.dietzel@unicaen.fr}

\begin{document}

\begin{abstract}
    In this article, we introduce a method to extend involutive nondegenerate set-theoretic solutions to the Yang--Baxter equation by means of equivariant mappings to graded modules, thus leading to the notion of a \emph{twisted extension}. Furthermore, we define coprime extensions of solutions and prove that each coprime extension of indecomposable solutions can be obtained as a suitable twisted extension. We then apply our results to obtain a full description of indecomposable solutions of size $pqr$, where $p,q,r$ are different primes, from a structure theorem of Cedó and Okniński. We close with some remarks on a cohomology theory for solutions developed by Lebed and Vendramin.

    We express our results in the language of cycle sets.
\end{abstract}

\maketitle

\section*{Introduction}

Let $V$ be a finite-dimensional vector space over $\C$ resp. $\C(q)$, the field of rational functions over $\C$. Given a linear automorphism $R: V^{\otimes 2} \to V^{\otimes 2}$, we say that $R$ satisfies the (\emph{linear}) \emph{Yang--Baxter equation}, if the following identity is satisfied in $\End(V^{\otimes 3})$:
\begin{equation} \label{eq:LYBE}
    R_{12}R_{23}R_{12} = R_{23}R_{12}R_{23} \tag{YBE1}
\end{equation}
where $R_{12}(u \otimes v \otimes w) = R(u \otimes v) \otimes w$ and $R_{23}(u \otimes v \otimes w) = u \otimes R(v \otimes w)$.

\cref{eq:LYBE} has first been studied by the physicists Baxter \cite{Baxter_YB} and Yang \cite{Yang_YB} in the context of statistical physics and integrable systems. Later, Turaev discovered that solutions to the (linear) Yang--Baxter equation give rise to link invariants \cite{Turaev_links}, such as the Jones polynomial, therefore providing a framework for both the unification of known link invariants and the construction of novel ones. Therefore, an understanding of the Yang--Baxter equation and its solutions sheds a new light on the theory of knots.

However, solutions to \eqref{eq:LYBE} are, in general, hard to obtain as their construction requires sophisticated tools, such as the representation theory of quantum groups. Furthermore, Drinfeld \cite{Drinfeld_Problems} noted that most (classical) linear solutions are essentially deformations of the \emph{flip} solution $R(v \otimes w) = w \otimes v$. This motivated him to ask for a systematic study of \emph{set-theoretic solutions} to the Yang--Baxter equation which correspond to linear solutions that permute a basis of $V^{\otimes 2}$, with the goal of obtaining novel families of solutions by deformation.

Let $X$ be a finite set and $r: X \times X \to X \times X$ a bijection. The pair $(X,r)$ is then said to be a \emph{set-theoretic solution to the Yang--Baxter equation} if the following functional equation for maps $X^3 \to X^3$:
\begin{equation} \label{eq:YBE}
    r_{12}r_{23}r_{12} = r_{23}r_{12}r_{23} \tag{YBE2}
\end{equation}
where $r_{ij}: X^3 \to X^3$ is the mapping that acts as $r$ on the coordinates $i,j$ and as identity on the others.

It has quickly been discovered by Etingof, Schedler and Soloviev \cite{ESS_YangBaxter} and Gateva-Ivanova and van den Bergh \cite{GIVdB_IType} that set-theoretic solutions can be extraordinarily well be analyzed with group-theoretic tools. This observation established the research on set-theoretic solutions as a field on its own that is, as of today, vivid and flourishing. Quickly, new tools, such as cycle sets \cite{Rump_Decomposition} and braces \cite{Rump_braces} have been introduced into this field and connections have been drawn to other fields of algebra, such as Garside theory \cite{Chouraqui_Garside}.

The algebraic approach to set-theoretic solutions is most efficient after imposing some restrictions on $r$. Writing $r(x,y) = (\lambda_x(y),\rho_y(x))$, we call a set-theoretic solution to the Yang--Baxter equation \emph{nondegenerate} if the maps $\lambda_x,\rho_x$ are bijective for all $x \in X$. In this work, solutions will always be non-degenerate, and we will furthermore assume that $X$ is \emph{involutive}, meaning that $r^2 = \id_{X \times X}$. These assumptions are necessary for an approach by means of nondegenerate cycle sets and left braces (of abelian type). However, it is notable that an algebraic theory can also be set up for solutions where $r$ is not necessarily involutive \cite{LYZ_YangBaxter,Guarnieri_Vendramin} and for solutions where $r$ is not necessarily bijective \cite{left_non_degenerate}.

From now on, we will always assume that $X$ is involutive and nondegenerate.

By a result of Rump \cite{Rump_Decomposition}, set-theoretic solutions to the Yang--Baxter equation are in bijective correspondence with simple algebraic structures, the \emph{nondegenerate cycle sets} (or simply: \emph{cycle sets}):

\begin{defn*}
    A \emph{cycle set} is a set $X$ with a binary operation $X \times X \to X$; $(x,y) \mapsto x \ast y$ such that the following three axioms are satisfied
    \begin{align}
        (x \ast y) \ast (x \ast z) & = (y \ast x) \ast (y \ast z) \quad (\forall x,y,z \in X), \tag{C1} \label{eq:cycloid_equation} \\
        \sigma_x: X \to X ; & \quad  y \mapsto x \ast y \quad \textnormal{ is bijective for all } x \in X \tag{C2}
    \end{align}
    If furthermore,
    \begin{align}
        \textnormal{the \emph{square map}} \quad \Sq: X \to X ; & \quad x \mapsto x \ast x \quad \textnormal{ is bijective,} \tag{C3}
    \end{align}
    a cycle set $X$ is \emph{nondegenerate}.
\end{defn*}

\begin{convention*}
    Throughout this article, we will implicitly assume that a cycle set is always nondegenerate and finite, unless we explicitly say otherwise. As usual, we will most of the time refer to a cycle set $(X,\ast)$ by its underlying set $X$.
\end{convention*}

With each set-theoretic solution $r(x,y) = (\lambda_x(y),\rho_y(x))$, one can associate a cycle set structure on $X$ by putting $x \ast y = \lambda_x^{-1}(y)$. On the other hand, for each cycle set $X$, there is a unique set-theoretic solution on $X$ that satisfies $r(x \ast y, x) = (y \ast x,y)$. By these assignments, cycle sets and solutions are in bijective correspondence.

A central problem in the theory of set-theoretic solutions is the classification of solutions. One way to approach this problem is over the classification of cycle sets by size. However, a full classification by size has only been obtained for cycle sets of size $\leq 10$ by Akgün, Mereb and Vendramin \cite{AMV_Cyclesets}, by means of a novel computational approach. The results of their computations suggest that the number of cycle sets of size $n$ grows so fast in $n$ that a full classification of cycle sets of slightly larger sizes is likely out of reach. Therefore, it is reasonable to restrict the issue of classification to a smaller class of cycle sets, the \emph{indecomposable} cycle sets, that is, cycle sets $X$ with \emph{no} nontrivial partition $X = X_1 \sqcup X_2$ such that $X_i \ast X_i = X_i$ ($i=1,2$). Translated to set-theoretic solutions, this means that $r(X_i \times X_j) = X_j \times X_i$ ($i,j \in \{1,2 \}$) for a partition $X = X_1 \sqcup X_2$ implies $X_1 = \emptyset$ or $X_2 = \emptyset$.

The restriction to indecomposable cycle sets is reasonable insofar as they can be considered as building blocks for general general cycle sets, in that they can be recursively assembled by means of a matching process to finally obtain all cycle sets. In fact, in the indecomposable case, there are several classification results by size:

\begin{enumerate}
    \item Etingof, Schedler and Soloviev classified indecomposable cycle sets of size $p$, where $p$ is a prime \cite{ESS_YangBaxter}.
    \item The results of Akgün, Mereb and Vendramin encompass the classification of indecomposable cycle sets of size $\leq 10$ \cite{AMV_Cyclesets}.
    \item Using a theorem of Cedó and Okniński on cycle sets of squarefree size \cite{CO_SquarefreeIndecomposable}, the classification of indecomposable cycle sets of size $pq$, where $p,q$ are different primes, follows from a result of Jedli\v{c}ka and Pilitowska \cite{jedlicka_pilitowska}.
    \item In collaboration with Properzi and Trappeniers, the author of this article has recently obtained a classification of indecomposable cycle sets of size $p^2$, where $p$ is a prime \cite{drz_psquare}.
\end{enumerate}

The results of Cedó and Okniński are the motivation for this article. Defining the \emph{retraction} of a cycle set as the quotient $X_{\ret} = X/\sim$ where $\sim$ is the equivalence relation given by $x \sim y \Leftrightarrow \sigma_x = \sigma_y$, one can define a cycle set structure on $X_{\ret}$ by putting $[x] \ast [y] = [x \ast y]$. Under this structure, the canonical quotient map $\ret: X \to X_{\ret}$ becomes a homomorphism of cycle sets. In \cite{CO_SquarefreeIndecomposable}, Cedó and Okniński prove that if $X$ is a cycle set of squarefree size $> 1$, then $|X_{\ret}| < |X|$ and $X_{\ret}$ is again a cycle set of squarefree size. Therefore, $X$ is an \emph{extension} of $X_{\ret}$ which opens up the possibility to parametrize $X$ in terms of dynamical cocycles over its retraction $X_{\ret}$, as has been discovered by Vendramin \cite{Vendramin_extension}. Lebed and Vendramin \cite{Lebed_Vendramin_Homology} refined this approach and showed that certain extensions can be parametrized by $2$-cocycles with values in an abelian group.

The Cedó-Okniński results imply that indecomposable cycle sets of squarefree size can be obtained by a recursive extension process from smaller ones of their kind. However, direct computation of dynamical cocycles in order to extend solutions is difficult and messy. In order to circumvent this problem, we develop in this article a method to parametrize a certain class of extensions that we call \emph{coprime}. It will be shown that all coprime extensions of an indecomposable cycle set can be obtained as \emph{twisted extensions} by maps into graded modules.
This method enables us to recover an indecomposable cycle set $X$ from $X_{\ret}$, whenever the size of the \emph{permutation group} $\G(X_{\ret})$ is coprime to the size of the fibers of $X \twoheadrightarrow X_{\ret}$ (\cref{sec:coprime_extensions}). We then develop a method to check if a twisted extension is indecomposable and describe the permutation group whenever the resulting extension is coprime (\cref{sec:permutation_groups}). We then apply this method to give a full description of indecomposable cycle sets of size $pqr$ in terms of multiplicative characters whenever $p,q,r$ are different primes (\cref{sec:size_pqr}). We close with a brief discussion of Lebed-Vendramin cohomology.

Due to the highly nontrivial descriptions obtained for the cycle sets of size $pqr$ and multipermutation level $3$ (see \cref{thm:pqr_with_uniconnected_retraction} and \cref{thm:pqr_with_non_uniconnected_retraction}), we refrain from an explicit description of the corresponding set-theoretic solutions.

\section{Preliminaries}

In this section, we collect basic results on cycle sets and braces.

Having already defined cycle sets in the introduction, we recall the exact definition of indecomposability.

\begin{defn}[{\cite[Definition 2.5.]{ESS_YangBaxter}, \cite{Rump_Decomposition}}]
    A cycle set $X$ is called \emph{indecomposable} if for each partition $X = X_1 \sqcup X_2$ into sub-cycle sets, we have either $X_1 = \emptyset$ or $X_2 = \emptyset$. 
\end{defn}

We will rely a lot on the following fundamental fact:

\begin{pro} \label{pro:size_of_image_divides_size_of_domain}
    Let $f: Y \twoheadrightarrow X$ be a surjective homomorphism of finite, indecomposable cycle sets. Then all fibers $f^{-1}(x)$ ($x \in X$) have the same size. In particular, $|X|$ divides $|Y|$.
\end{pro}

\begin{proof}
    See \cite[Lemma 3.3]{CO_SquarefreeIndecomposable}.
\end{proof}

A remarkable property of finite cycle sets is that a significant proportion of them has repeated permutations. This gives rise to the notion of the \emph{retraction} of cycle sets:

\begin{defn}[{\cite[§3.2.]{ESS_YangBaxter}, \cite{Rump_Decomposition}}]
    Let $X$ be a cycle set. Then the \emph{retraction} of $X$ is the quotient $X_{\ret} = X/\sim$, where $\sim$ is the equivalence relation given by $x \sim y \Leftrightarrow \sigma_x = \sigma_y$, together with the cycle set operation $[x] \cdot [y] = [x \cdot y]$.
\end{defn}
That $X_{\ret}$ indeed is a cycle set under the given operation is proven in \cite[Lemma 2]{Rump_Decomposition}.

A cycle set can often be retracted more than once which motivates the notion of \emph{multipermutation level} \cite[Definition 3.1.]{ESS_YangBaxter}: define recursively $X^{(0)} = X$, and for $n \geq 0$, put $X^{(n+1)} = X^{(n)}_{\ret}$. Then the \emph{multipermutation level} of $X$ is defined as
\[
\mpl(X) = \begin{cases}
    \min \{n \geq 0 : |X^{(n)}| = 1 \} & \exists n \geq 0: |X^{(n)}| = 1, \\
    \infty & \textnormal{else}.
\end{cases}
\]
In this article, we will most of the time denote $X_{\ret}$ as $X^{(1)}$.

An important invariant of a cycle set is its \emph{permutation group} which is the subgroup
\[
\G(X) = \genrel{\sigma_x}{x \in X} \leq \Sym_X
\]
where, as in the axioms, $\sigma_x(y) = x \ast y$. Note that $X$ is indecomposable if and only if $\G(X)$ acts transitively on $X$.

It turns out that $\G(X)$ can be enriched by another group structure that turns $\G(X)$ into what is called a left brace of abelian type. As we will nowhere delve into the very rich theory of skew braces, we restrict the definition of a \emph{brace} to this special case. In this article, we use the definition of Cedó, Jespers and Okniński but remark that braces of abelian type go back to the work of Rump \cite{Rump_braces}. In the following, we collect some basic results on braces, most of which can be found in \cite{CJO_Braces}. If not, we explicitly state the reference.

\begin{defn}
    A \emph{brace} is a triple $B = (B,+,\circ)$ such that $B^+ = (B,+)$ is an abelian group, the \emph{additive group} of $B$, $B^{\circ} = (B,\circ)$ is a group - the \emph{multiplicative group} of $B$ - and the identity
    \[
    a \circ (b + c) = a \circ b - a + a \circ c
    \]
    is satisfied for all $a,b,c \in B$.
\end{defn}

Homomorphisms of braces are defined as homomorphisms of both the additive and the multiplicative group structures. Everything that comes with the concept of homomorphism (isomorphism,\ldots) is defined in the obvious way.

In general, we will suppress the $\circ$ and abbreviate $a \circ b = ab$. Note that it can be shown that $B^+$ and $B^{\circ}$ share the same identity which we denote as $0$.

Let us now discuss the algebraic structure of braces needed for this work! First of all, define the \emph{$\lambda$-action} of a brace by $\lambda_a(b) = ab - a$ ($a,b \in B$). As it satisfies the equations
\[
\lambda_a(b+c) = \lambda_a(b) + \lambda_a(c); \quad \lambda_{ab} = \lambda_a \circ \lambda_b,
\]
the map
\begin{align*}
    \lambda: (B, \circ) & \to \Aut(B^+) \\
    a & \mapsto \lambda_a
\end{align*}
is a well-defined homomorphism of groups. In particular, $B^+$ becomes a $B^{\circ}$-module.

A brace can be seen as a \emph{ring-like} structure, accordingly there is a notion of (left) ideal:

\begin{defn}
    Let $B$ be a brace. A \emph{left ideal} of $B$ is a subset $I \subseteq B$ such that $I$ is a subgroup of $B^+$ that is invariant under the $\lambda$-action, in that $\lambda_a(I) = I$ for all $a \in I$.

    If furthermore, $I$ is a normal subgroup of $B^{\circ}$, then $I$ is called an \emph{ideal} of $B$.
\end{defn}
Note that each left ideal $I \subseteq B$ is also a multiplicative subgroup of $B$. As in ring theory, the quotient
\[
B/I = \{ a + I : a \in B  \} =  \{ a \circ I : a \in B  \}
\]
can be formed and if $I$ is an ideal, $B/I$ is again a brace under the operations inherited by $B$.

A distinguished ideal of a brace $B$, that will play an important role throughout this work is the \emph{socle}
\[
\Soc(B) = \{ a \in B: \forall b \in B: ab = a+b \} = \ker(\lambda).
\]

We quickly reprove a well-known result about the socle that will be useful later:

\begin{pro} \label{pro:lambda_on_socle_is_conjugation}
    Let $B$ be a a brace. Then for $a \in B$, and $b \in \Soc(B)$, we have
    \[
    \lambda_a(b) = {}^ab.
    \]
\end{pro}

\begin{proof}
    Let $a, b$ be as in the statement. Then normality of $\Soc(B)$ implies that ${}^ab \in \Soc(B)$ either. Therefore,
    \[
    ab = {}^aba \Rightarrow a + \lambda_a(b) = {}^ab + \underbrace{\lambda_{{}^ab}(a)}_{=a} \Rightarrow a + \lambda_a(b) = {}^ab + a \Rightarrow \lambda_a(b) = {}^ab.
    \]
\end{proof}

\begin{convention*}
    For an integer $n$, we denote by $\pi(n)$ the set of primes dividing $n$.
\end{convention*}

An important kind of left ideal that we will consider a lot in this article is the following:

\begin{defn}
    Let $B$ be a finite brace and $\pi$ a set of primes. Then the \emph{$\pi$-primary component} of $B$ is defined as
    \[
    B_{\pi} = \{ a \in B: \pi(o_+(a)) \subseteq \pi \},
    \]
    the (additive) Hall-$\pi$-subgroup of $B_{\pi}$, where $o_+(a)$ denotes the order of $a$ in $B^+$.
\end{defn}
Note that $B_{\pi}$ is a characteristic subgroup in $B^+$ and thus, a left ideal in $B$.

We will furthermore rely on the notion of a \emph{semidirect product} of braces. As it is sufficient for our needs, we will restrict to \emph{external} semidirect products.

\begin{defn}
Let $B$ be a brace and let $B_1,B_2 \leq B$ be left ideals where $B_1$ is an ideal. If the additive group of $B^+$ decomposes as a direct sum $B^+ = B_1 \oplus B_2$, we call $B$ a (n \emph{external}) \emph{semidirect product} of $B_1,B_2$. In such a case, we write $B = B_1 \rtimes B_2$ (or $B_2 \ltimes B_1$, depending on the factor that we want to put the emphasis on.
\end{defn}

The relevance of braces for the theory of cycle sets lies in the fact that the permutation group $\G(X)$ of a cycle set $X$ carries a unique additive group structure $\G(X)$ such that $(\G(X),+,\circ)$ is a brace: the multiplication $\circ$ is the usual composition of permutations and, defining the element $\lambda_x = \sigma_x^{-1} \in \G(X)$ for $x \in X$, there is a unique brace structure on $\G(X)$ with
\[
\lambda_x + \lambda_y = \lambda_x \circ \lambda_{\lambda_x^{-1}(y)}
\]
for all $x,y \in X$ or, equivalently, $\sigma_x^{-1}(y) = \lambda_x(y)$, the $\lambda$-operation of the brace.

Note that, under this brace structure, $\lambda_{x \ast y} = \lambda_{\lambda_x}^{-1}(\lambda_y)$ is satisfied for $x,y \in X$.

\begin{convention*}
    $\Sym_X$ and all elements therein act from the left. For an element $g \in \G(X) \leq \Sym_X$, the notation $\lambda_g$ has a threefold meaning: it refers to the action of $g$ on $X$, to the $\lambda$-action of $g$ on $\G(X)$ and, as $\G(X)$ is a permutation group, $\lambda_g$ ultimately refers to the element $g$ itself.
    There is no danger of confusion, though, as the $\lambda$-action on $\G(X)$ is compatible with the $\lambda$-action on $X$ and it will always be clear which set is acted upon by $\G(X)$. However, we will avoid to equate $g = \lambda_g$ when no actions on other elements are involved.

    Furthermore, for an element $g \in \G(X)$, we will write $\sigma_g = \lambda_g^{-1}$ whenever it is more convenient to do so.
\end{convention*}

It is important to realize that the construction of the permutation group of a cycle set is \emph{not} functorial without restriction. However, it becomes so after restricting to surjective homomorphisms:

\begin{thm}
    Let $\varphi: Y \twoheadrightarrow X$ be a surjective homomorphism of cycle sets. Then there is a unique brace homorphism $\G(\varphi): \G(Y) \to \G(X)$ such that $\G(\varphi)(\lambda_y) = \lambda_{\varphi(y)}$ for all $y \in \G(Y)$.
\end{thm}

\begin{proof}
    See \cite[p.5]{rump_kanrar}.
\end{proof}

The following result will be fundamental for our investigations, and due to its ubiquity, we will not refer to it all the time:

\begin{thm} \label{thm:soc_is_ker_g_ret}
    Let $\ret: X \twoheadrightarrow X^{(1)}$ be the retraction morphism, then $\ker(\G(\ret)) = \Soc(\G(X))$. In particular, if $f: Y \twoheadrightarrow X$ is an epimorphism such that $\ret: Y \twoheadrightarrow Y_{\ret}$ factors through $f$, then $\ker(\G(f)) \subseteq \Soc(\G(Y))$.
\end{thm}

\begin{proof}
    \cite[Lemma 6.1.]{CJO_Braces}
\end{proof}

\section{Construction of coprime extensions of cycle sets} \label{sec:coprime_extensions}

\begin{convention*}
    Given a group $G$, acting from the left on a set $X$ via the operation $(g,x) \mapsto g \cdot x$, we denote the orbit of an element $x \in X$ under the action of $G$ by
    \[
    \mathcal{O}_G(x) = \{ g \cdot x : g \in G \}.
    \]
\end{convention*}

\begin{defn}
    Let $X,Y$ be cycle sets where $X$ is indecomposable. We call a surjective homomorphism $f: Y \twoheadrightarrow X$ a \emph{constant extension} (\emph{extension}, for short) if for all $x,y \in Y$, we have the implication $f(x) = f(y) \Rightarrow \sigma_x = \sigma_y$. If $Y$ is also indecomposable, we call $f$ an \emph{extension of indecomposable cycle sets}.

    An extension of indecomposable cycle sets $f: Y \twoheadrightarrow X$ is called \emph{coprime} if $\gcd\left(|\G(X)|, \frac{|Y|}{|X|}\right) = 1$.
\end{defn}

Note that by \cref{pro:size_of_image_divides_size_of_domain}, $\frac{|Y|}{|X|} = |f^{-1}(x)|$ for all $x \in X$.

We will also need to talk about \emph{equivalence} of extensions:

\begin{defn}
    Let $f_i: Y_i \twoheadrightarrow X$ ($i =1,2$) be two extensions of a cycle set $X$. We call $f_1,f_2$ \emph{equivalent} if there is an isomorphism $\iota: Y_1 \to Y_2$ such that $f_1 = f_2 \circ \iota$.
\end{defn}

The following proposition follows directly from \cref{thm:soc_is_ker_g_ret}:

\begin{pro} \label{pro:kernel_of_extensions_is_in_socle}
    Let $f: Y \twoheadrightarrow X$ be an extension of cycle sets, then $\ker(\G(f)) \subseteq \Soc(\G(Y))$.
\end{pro}

Before giving the construction of general coprime extensions of indecomposable cycle sets, we need to introduce and recall some concepts:

Let $X$ be a set and $A$ an abelian group. It is easy to see that $A^X$, the group of tuples $a = (a_x)_{x \in X}$ under coordinatewise addition, acts on $X \times A$ by
\[
    a + (x,b) = (x,b + a_x).
\]
We want to give a description of this action that is equivariant with respect to an action of $G$ on $X$.

To this end, let $X$ be a set that is acted upon by a group $G$ by means of an operation denoted by $(g,a) \mapsto g \cdot x$ ($g \in G, x \in X$). Define an \emph{$X$-graded $G$-module} as an abelian group $A$ with a decomposition into components
\[
A = \bigoplus_{x \in X} A_x,
\]
and an action by automorphisms $\rho: G \to \Aut(A)$, also denoted by $g \cdot a = (\rho(g))(a)$, that respects the grading in the sense that
\[
g \cdot A_x \subseteq A_{g \cdot x} \quad (g \in G, x \in X).
\]
We write the elements of $A$ as tuples $(a_x)_{x \in X}$ with $a_x \in A_x$. Furthermore, given two $X$-graded $G$-modules $A,B$, we call an additive homomorphism $f: A \to B$ a \emph{homomorphism of $X$-graded $G$-modules} if $f$ is $G$-equivariant, i.e. $f(g \cdot a) = g \cdot f(a)$ for all $g \in G$, $a \in A$, and if $f$ respects the grading by $X$, meaning that $f(A_x) \subseteq B_x$ for all $x \in X$. All notions that come along with homomorphisms (isomorphism, \ldots) are defined in the obvious way.

Given an $X$-graded $G$-module $A$, we define the \emph{scattering} of $A$ as the set
\[
A^{\sca} = \bigsqcup_{x \in X} \{ x \} \times A_x
\]
which has a canonical $G$-action given by
\[
g \cdot (x,a) = (g \cdot x, g \cdot a).
\]

It can easily be shown that $A$ acts on $A^{\sca}$ in a $G$-equivariant way:

\begin{pro}
    Let $A$ be an $X$-graded $G$-module, then the binary operation
    \begin{align*}
        A \times A^{\sca} & \to A^{\sca} \\
        (a, (x,b)) & \mapsto a + (x,b) = (x,b+a_x),
    \end{align*}
    is an action of $A$ on $A^{\sca}$ that is $G$-equivariant in the sense that
    \begin{equation} \label{eq:action_on_a_sc}
    g \cdot (a + (x,b)) = g \cdot a + g \cdot (x,b)
    \end{equation}
    for $g \in A$, $a \in A$, $(x,b) \in A^{\sca}$.
\end{pro}

\begin{proof}
    As $A$ acts componentwisely on $A^{\sca}$ by regular actions of the respective factors $A_x$, it is easy to check this is indeed a group action. For \cref{eq:action_on_a_sc}, we calculate
    \[
    g \cdot (a + (x,b)) = g \cdot (x, b + a_x) = (g \cdot x, g \cdot b + g \cdot a_x) = g \cdot a + g \cdot (x, b).
    \]
\end{proof}

\begin{defn}
    Let $X$ be a cycle set and let $A$ be an $X$-graded group with respect to the action of $\G(X)$ on $X$. Let $\Phi: X \times X \to A$ be a map such that
    \begin{align}
        \Phi(x,y) & \in A_y \label{eq:phixy_in_ay} \\
        \Phi(\lambda_g(x), \lambda_g(y)) & = g \cdot \Phi(x,y). \label{eq:phi_equivariance}
    \end{align}
    for $x,y \in X$, $g \in \G(X)$. Then we define the \emph{twisted extension} $X \otimes_{\Phi}A$ as the set $A^{\sca}$ with the binary operation
    \begin{equation} \label{eq:cycle_operation_by_phi}
    (x,a) \ast (y,b) = \lambda_x^{-1} \cdot(y, b + \Phi(x,y)) = (x \ast y, \lambda_x^{-1} \cdot(b + \Phi(x,y))),
    \end{equation}
    together with the projection $\pr: X \otimes_{\Phi}A \twoheadrightarrow X$; $(x,a) \mapsto x$.
\end{defn}

\begin{pro}
    $\pr: X \otimes_{\Phi} A \twoheadrightarrow X$ defines an extension of cycle sets.
\end{pro}

\begin{proof}
    First note that $\lambda_x \lambda_{x \ast y} = \lambda_y \lambda_{y \ast x}$, which is a mere reformulation of \cref{eq:cycloid_equation} in terms of the $\lambda$-action. Using that, we calculate
    \begin{align*}
        ((x,a) \ast (y,b)) \ast ((x,a) \ast (z,c)) & =  (x \ast y, \lambda_x^{-1} \cdot (b + \Phi(x,y))) \ast (x \ast z, \lambda_x^{-1} \cdot (c + \Phi(x,z))) \\
        & = ((x \ast y) \ast (x \ast z), \lambda_{x \ast y}^{-1} \cdot ( \lambda_x^{-1} \cdot (c + \Phi(x,z)) + \underbrace{\Phi(x \ast y, x \ast z)}_{= \lambda_x^{-1} \cdot \Phi(y,z)}) ) \\
        & = ((x \ast y) \ast (x \ast z), (\lambda_x \lambda_{x \ast y})^{-1} \cdot (c + \Phi(x,z) + \Phi(y,z))) \\
        & = ((y \ast x) \ast (y \ast z), (\lambda_y \lambda_{y \ast x})^{-1} \cdot (c + \Phi(x,z) + \Phi(y,z)) ) \\
        & = ((y,b) \ast (a,x)) \ast ((y,b) \ast (z,c)).
    \end{align*}
    Furthermore, it is quickly seen from the non-degeneracy of $X$ that the square map $\Sq(x,a) = (\Sq(x), \lambda_x^{-1} \cdot (a + \Phi(x,x)))$ is bijective as well, so $X \otimes_{\Phi}A$ is indeed a non-degenerate cycle set. Furthermore, it is clear that $\pr$ is a homomorphism of cycle sets. As moreover, $(x,a) \ast (y,b) = \lambda_x^{-1} \cdot(y, b + \Phi(x,y))$ is independent of $a$, the map $\pr$ is indeed an extension of cycle sets.
\end{proof}

Note that different maps $\Phi$ can define isomorphic twisted extensions. This will be needed later.

\begin{pro} \label{pro:isomorphisms_of_A_give_isomorphic_extensions}
    Let $\pr: X \otimes_{\Phi} A \twoheadrightarrow X$ be a twisted extension and let $f: A \to A$ be an isomorphism of $X$-graded $G$-modules. Then, with $\Phi' = f \circ \Phi$, the map $\pr': X \otimes_{\Phi'}A \twoheadrightarrow X$ is also a twisted extension, and the map
    \[
    g: X \otimes_{\Phi}A \to X \otimes_{\Phi'}A ; \quad (x,a) \mapsto (x,f(a))
    \]
    is an equivalence between $\pr$ and $\pr'$.
\end{pro}

\begin{proof}
    It is quickly checked that $\Phi'$ still satisfies \cref{eq:phixy_in_ay} and \cref{eq:phixy_in_ay}, so $\Phi'$ again defines a twisted extension. The equivalence property follows from a short calculation.
\end{proof}

\begin{cor}\label{pro:multiples_of_phi_give_isomorphic_extensions}
    Let $\pr: X \otimes_{\Phi}A \twoheadrightarrow X$ be a twisted extension and let $k$ be an integer such that $\gcd(k,|A|) = 1$, then the extension $\pr': X \otimes_{k \cdot \Phi} A \twoheadrightarrow X$ is equivalent to $\pr$.
\end{cor}

\begin{proof}
    Under the given conditions, $A \to A$; $a \mapsto k \cdot a$ is an isomorphism of $X$-graded $G$-modules.
\end{proof}

In the special case when then $X$-graded $\G(X)$-module $A$ is a permutation module, i.e. there is an isomorphism $A \cong B^X$ for some abelian group $B$, where $(g \cdot f)(x) = f(\lambda_g^{-1}(x))$, the action on $A^{\sca} = X \times B$ simplifies to $g \cdot (x,a) = (\lambda_g(x),a)$, so we can obtain twisted extensions in a simpler way: in this case, they can be described by a map $\Gamma: X \times X \to B$ that satisfies the invariance property
\begin{equation} \label{eq:equivariance_for_parallel_extensions}
    \Gamma(\lambda_g(x),\lambda_g(y)) = \Gamma(x,y) \ \textnormal{for } x,y \in X,\ g \in \G(X).
\end{equation}

In this special case, \cref{eq:cycle_operation_by_phi} turns into the following operation on $X \times B$:
\begin{equation} \label{eq:cycle_operation_parallel_extension}
(x,a) \ast (y,b) = (x \ast y, b + \Gamma(x,y)).    
\end{equation}
We denote the resulting cycle set by $X \times_{\Gamma} B$ and call it a \emph{parallel extension} of $X$ by $B$ (via $\Gamma$).

Despite we try to make our approach as independent of a base point as possible, it will occasionally be useful to make use of the \emph{fundamental group} of a cycle set \cite{rump_uniconnected}:

\begin{defn}
    Let $X$ be an indecomposable cycle set and $x_0 \in X$. Then the \emph{fundamental group} of $X$ with \emph{base point} $x_0$ is the subgroup
    \[
    \Pi_1(X,x_0) = \{ g \in \G(X) : \lambda_g(x_0) = x_0 \} \leq \G(X).
    \]
\end{defn}

\begin{rem}
    The fundamental group $\Pi_1(X,x_0)$ also plays a role in Rump's covering theory \cite{Rump_coverings}, as its subgroups are in correspondence with \emph{strong coverings} i.e. surjective homomorphisms $f: (Y,y_0) \twoheadrightarrow (X,x_0)$  where $Y$ is an indecomposable cycle set with base point $y_0$ and $\G(f): \G(Y) \to \G(X)$ is an isomorphism. It can be shown that $\frac{|Y|}{|X|}$ divides $|\G(X)|$ in these cases, so no non-trivial strong covering is a coprime extension. Surprisingly, the fundamental group $\Pi_1(X,x_0)$ will nevertheless play a significant role in the description of coprime extensions.
\end{rem}

As we are mostly concerned with indecomposable cycle sets, it will sometimes be favorable to work with the following version of $\G(X)$-equivariance:

Let $X$ be indecomposable. Choosing a point $x_0 \in X$, we consider the map $\Phi^0(x) = \Phi(x_0,x)$. Then \cref{eq:phixy_in_ay} tells us that 
\begin{equation} \label{eq:phi0y_in_ay}
    \Phi^0(y) \in A_y \quad \textnormal{for all } y \in X.
\end{equation}

Putting a $g \in \Pi = \Pi_1(X,x_0)$ in \cref{eq:phi_equivariance}, equivariance furthermore specializes to
\begin{equation} \label{eq:equivariance_for_pi1}
    \Phi^0(\lambda_g(y)) = g \cdot \Phi^0(y).
\end{equation}

On the other hand, given a mapping $\Phi^0: X \to A$ such that \cref{eq:phi0y_in_ay} is satisfied and \cref{eq:equivariance_for_pi1} holds for all $g \in \Pi$, a mapping $\Phi: X \times X \to A$ satisfying \cref{eq:phixy_in_ay} and \cref{eq:phi_equivariance} is easily constructed by
\begin{equation}
    \Phi(\lambda_g(x_0),y) = g \cdot \Phi^0(\lambda_g^{-1}(y)).
\end{equation}

Therefore, we can write the operation on $X \otimes_{\Phi} A$ as
\begin{equation} \label{eq:twisted_extensions_expressed_by_phi0}
    (x,a) \ast (y,b) = \left(x \ast y, (\lambda_x^{-1}g) \cdot (b + \Phi^0(\lambda_g^{-1}(y))) \right)
\end{equation}
where $g \in \G(X)$ satisfies $\lambda_g(x_0)=x$.

We see from the above discussion that, given an $X$-graded $\G(X)$-module $A$, any map $\Phi^0: X \to A$ satisfying \cref{eq:phi0y_in_ay} and \cref{eq:equivariance_for_pi1} defines an equivariant map $\Phi: X \times X \to A$ and therefore can be used to construct the cycle set $X \otimes_{\Phi} A$ by means of \cref{eq:twisted_extensions_expressed_by_phi0}.

In the special case of a parallel extension $X \times_{\Gamma}B$, we similarly define $\Gamma^0(x) = \Gamma(x_0,x)$ which satisfies $\Gamma^0(\lambda_g(x)) = \Gamma^0(x)$ whenever $g \in \Pi_1(X,x_0)$ and $\Gamma$ can be recovered from such a map by $\Gamma(x,y) = \Gamma^0(\lambda_g^{-1}(y))$ for $g \in \G(X)$ with $\lambda_g(x_0)=y$.

The expression in \cref{eq:twisted_extensions_expressed_by_phi0} now becomes
\begin{equation} \label{eq:parallel_extensions_expressed_by_gamma0}
    (x,a) \ast (y,b) = \left(x \ast y, b + \Gamma^0(\lambda_g^{-1}(y))) \right)
\end{equation}
where $g \in \G(X)$ satisfies $\lambda_g(x_0)=x$.

We will later see that an understanding of the orbits of $\Pi_1(X,x_0)$ on $X$ will be of high relevance for our classification.

We now state one main theorem of this section:

\begin{thm} \label{thm:coprime_extensions_are_twisted_extensions}
    Let $f: Y \twoheadrightarrow X$ be a coprime extension of indecomposable cycle sets. Then $f$ is equivalent to a twisted extension $\pr: X \otimes_{\Phi} A \to X$.
\end{thm}

The following lemma is a special case of a more general, unpublished, observation that was made in former collaboration with Silvia Properzi and Senne Trappeniers. In this special case, we give a simplified proof that makes use of a fixed point argument similar to the one discovered by Trappeniers.

\begin{lem} \label{lem:primes_in_an_extension}
    Let $f: Y \twoheadrightarrow X$ be an extension of indecomposable cycle sets and $K = \ker \G(f)$. Then $\pi(|K|) \subseteq \pi \left( \frac{|Y|}{|X|} \right)$.
\end{lem}

\begin{proof}
    Take any prime $p \not\in \pi \left(\frac{|Y|}{|X|}\right)$ and let $K_p$ the unique additive $p$-Sylow subgroup of $K$ which is also the unique multiplicative $p$-Sylow subgroup of $K$ as $K$ is a trivial ideal. Therefore, $K_p$ is characteristic in $K^{\circ}$. $K$ leaves the fibers $f^{-1}(x)$ ($x \in X$) setwise fixed. As $p \nmid |f^{-1}(x)| = \frac{|Y|}{|X|}$, the action of $K_p$ has a fixed point on $Y$, say $y$. Therefore $K_p \leq \G(Y)_y$, the stabilizer of $y$. But $K_p$ is characteristic in $K^{\circ}$ which is normal in $\G(Y)^{\circ}$. Therefore, $K_p$ is normal in $\G(Y)^{\circ}$ and therefore, acts trivially on $Y$. This implies that $K_p = 1$.
\end{proof}

Recall furthermore that the \emph{holomorph} of a group $G$ is defined as the semidirect product (of groups)
\[
\Hol(G) = G \rtimes \Aut(G)
\]
Where $\Aut(G)$ acts by the respective automorphisms on the first factor. There is a canonical action of $\Hol(G)$ by affine maps on $G$ via $(g,\alpha) \cdot h = g \alpha(h)$. We can now state the following lemma:

\begin{lem} \label{lem:coprime_affine_actions_have_fixed_points}
    Let $A$ be a finite group and let $G \leq \Hol(A)$ be a subgroup with $\gcd(|G|,|A|) = 1$. If $A$ is solvable or $G$ is solvable, $G$ has a fixed point on $A$.
\end{lem}

\begin{proof}
    Consider the group $H = AG \leq \Hol(A)$ and the projection $\pi: \Hol(A) \twoheadrightarrow \Aut(A)$; $(a,\alpha) \mapsto \alpha$. As $|G|$ is coprime to $|A|$, the map $\pi$ maps $G$ isomorphically onto its image $\pi(G)$. The section $s: \Aut(A) \to \Hol(A)$; $\alpha \mapsto (0,\alpha)$ now induces the section $G^{\prime} =(s\pi)(G)$ of $\pi(G)$ that obviously fixes the point $0 \in A$. By the Schur-Zassenhaus theorem \cite[§18.2]{huppert1983endliche}, $G = {}^gG^{\prime}$ for some $g \in \Hol(A)$, so $G$ fixes the point $g \cdot 0$.
\end{proof}

We can finally prove the first main result of this section:

\begin{proof}[Proof of \cref{thm:coprime_extensions_are_twisted_extensions}]
    Let $\pi = \pi(|\G(X)|)$, and consider the subgroups $\G(Y)_{\pi} \leq \G(Y)$ and $K = \ker(f) \trianglelefteq \G(Y)$. By \cref{lem:primes_in_an_extension} and the assumption of coprimality, $\gcd(|K|,|\G(Y)_{\pi}|)= 1$, therefore $\G(f)$ maps $\G(Y)_{\pi}$ isomorphically onto $\G(X)$ and furthermore, we have a semidirect product of braces $\G(Y) = K \rtimes \G(Y)_{\pi}$.

    By definition, $K$ fixes the fibers $\mathcal{F}_x = f^{-1}(x)$ ($x \in X$) setwise. Consider the orbit set $K \setminus Y$. As $K$ is normal in $\G(Y)$, the orbits $\Orb_K(y)$ ($y \in Y$) form blocks for the action of $\G(Y)$ on $Y$. As they have the same size and refine the partition of $Y$ by the family $(\mathcal{F}_x)_{x \in X}$, we see that $|K \setminus Y| = |X| \cdot m$ for some $m \mid |\mathcal{F}_x|$.
    
    On the other hand, the decomposition $\G(Y) = K \rtimes \G(Y)_{\pi}$ shows that $\G(Y)_{\pi}$ acts transitively on $|K \setminus Y|$, so $\pi(|K \setminus Y|) \subseteq \pi$. In particular, $\gcd(|K \setminus Y|, |\mathcal{F}_x|) = 1$, so $|K \setminus Y| = |X|$.
    
    This shows that $K$ acts transitively on each fiber. As $K \leq \Soc(K)$ by \cref{pro:kernel_of_extensions_is_in_socle}, we see that $K^{\circ}$ is abelian, so $K$ induces regular transitive actions of abelian groups on the fiber $\mathcal{F}_x$ that are pairwise conjugate. Denoting the factors $A_x = \mathrm{im}(K \to \Sym_{\mathcal{F}_x})$, we can consider $K$ as a subgroup of $A = \prod_{x \in X}A_x$ in a way that $K$ is a subdirect product of the $(A_x)_{x \in X}$.

    Denote by $s: \G(X) \hookrightarrow \G(Y)$ the canonical section that identifies $\G(X)$ with $\G(Y)_{\pi}$. Now pick a point $x_0 \in X$ and consider the fundamental group $\Pi_{x_0} = \Pi_1(X,x_0)$ resp. the subgroup $\Pi_{x_0}^{\prime} = s(\Pi_{x_0})$. Then $\Pi_{x_0}^{\prime}$ fixes the fiber $\mathcal{F}_{x_0}$ as a set and, as $K$ is normalized by $\Pi_{x_0}^{\prime}$, we deduce that the image of $\Pi_{x_0}^{\prime}$ in $\Sym(\mathcal{F}_{x_0})$ normalizes the regular transitive action of $A_{x_0}$ on $\mathcal{F}_{x_0}$. As $\gcd(|\Pi_{x_0}^{\prime}|,|A_{x_0}|) = 1$, \cref{lem:coprime_affine_actions_have_fixed_points} implies that $\Pi_{x_0}^{\prime}$ has a fixed point $y_0 \in \mathcal{F}_{x_0}$.

    By a standard arguments concerning group actions, we see that $\Pi_{\lambda_g(x_0)} = {}^g \Pi_{x_0}$. We furthermore see that $\lambda_{s(g)}(y_0)$ is a fixed point of $s(\Pi_{\lambda_g(x_0)})$ acting on the fiber $\mathcal{F}_{\lambda_g(x_0)}$.
    
    Now consider the orbit $X^{\prime} = \Orb_{\G(Y)_{\pi}}(y_0)$. As $\G(X)$ acts transitively on $X$, we have $X^{\prime} \cap \mathcal{F}_x \neq \emptyset$ for all $x \in X$. Furthermore, if $\{y_1,y_2 \} \in \mathcal{F}_x \cap X^{\prime}$ for some $x \in X$, then $y_2 = \lambda_{s(g)}(y_1)$ for some $g \in \G(X)$. As $\lambda_{s(g)}$ then fixes the fiber $\mathcal{F}_x$ setwise, we conclude that $g \in \Pi_x$. But we have seen that the elements in $X^{\prime} \cap \mathcal{F}_x$ are all fixed by $s(\Pi_x)$, so $y_1 = y_2$. Therefore, $|\mathcal{F}_x \cap X^{\prime}| =1 $ for all $x \in X$. For $x \in X$, define the point $y_x$ by $\{ y_x\} = X^{\prime} \cap \mathcal{F}_x$.

    Recalling that $A_x = \mathrm{im}(K \to \Sym_{\mathcal{F}_x})$ and that $K$ is normalized by $\G(Y)$, we get an action of $\G(X)$ by automorphisms of $A = \bigoplus_{x \in X} A_x$ by
    \begin{align*}
        \G(X) \times A_x & \to A_{\lambda_g(x)} \\
        (g,\lambda_k|_{\mathcal{F}_x}) & \mapsto g \cdot (\lambda_k|_{\mathcal{F}_x}) =(\lambda_{{}^{s(g)}k})|_{\mathcal{F}_{\lambda_g(x)}}.
    \end{align*}
    With this rule $A$ is an $X$-graded $G$-module.

    We can now use the fixed points $y_x$ ($x \in X$) to coordinatize $Y$ by means of the bijection
    \[
        \gamma: A^{\sca}  \to Y ; \quad (x,a)  \mapsto a + y_x.
    \]
    Here we suggestively use additive notation for the action of $A_x$ on $\mathcal{F}_x$. More precisely, for $a = \lambda_k|_{\mathcal{F}_x}$, we mean $a + y_x = \lambda_k(y_x)$. Note that $\gamma(x,0) = y_x$.

    Now $\gamma$ is, in fact, an isomorphism of $\G(X)$-actions: writing $a = \lambda_k|_{\mathcal{F}_x}$, we calculate for $g \in \G(X)$:
    \begin{align*}
        \lambda_{s(g)}(\gamma(x,a)) & = \lambda_{s(g)}(\lambda_k(y_x)) \\
        & = \lambda_{{}^{s(g)}k} ( \lambda_{s(g)}(y_x)) \\
        & = (g \cdot (\lambda_k|_{\mathcal{F}_x}))(y_{\lambda_g(x)}) \\
        & = \gamma(\lambda_g(x), g \cdot a) = \gamma(g \cdot (x,a)).
    \end{align*}

    Under the coordinatization provided by $\gamma$, we can therefore identify $Y = A^{\mathrm{sc}}$ for some $X$-graded $\G(X)$-module $A$.

    Due to the fact that $f: Y \to X$ is an extension, we know that $\lambda_{(x,a)}$ is only dependent on $x$. Moreover, the decomposition $\G(Y) = K \rtimes \G(X)_{\pi}$ tells us that for each $(x,a) \in Y$, there is a unique decomposition
    \[
    \lambda_{(x,a)} = k_x + s(\lambda_x)
    \]
    where $k_x \in K$. As $K \leq A$, there is a map $\Phi: X \times X \to A$ with $\Phi(x,y) \in A_y$ such that for all $x,y \in X$, we have 
    \[
    \lambda_{k_x}(y,a) = (y,a -\Phi(x,y)).
    \]
    As $K \subseteq \Soc(\G(Y))$, we can rewrite
    \[
    \lambda_{(x,a)} = k_x + s(\lambda_x) = k_x s(\lambda_x),
    \]
    therefore
    \begin{align*}
    (x,a) \ast (y,b) & = \lambda_{(x,a)}^{-1}(y,b) \\
    & = \lambda_{s(\lambda_x)}^{-1} \lambda_{k_x}^{-1}(y,b) \\
    & = \lambda_{s(\lambda_x)}^{-1} (y, b + \Phi(x,y)) \\
    & = \lambda_x^{-1} \cdot (y, b + \Phi(x,y)).
    \end{align*}

    We are left with proving equivariance: pick any $g \in \G(X)$ and make use of \cref{pro:lambda_on_socle_is_conjugation} to obtain that
    \begin{align*}
    \lambda_{(\lambda_g(x),a)} & = \lambda_{s(g)}(k_x + s(\lambda_x))  \\
    & = \lambda_{s(g)} (k_x) + \lambda_{s(g)}(s(\lambda_x)) \\
    & = {}^{s(g)}k_x + s({\lambda_{\lambda_g(x)}}).
    \end{align*}
    which implies $k_{\lambda_g(x)} = {}^{s(g)}k_x$ resp. $s(g)k_x = k_{\lambda_g(x)}s(g)$. Therefore, for $(y,b) \in Y$,
    \begin{align*}
        \lambda_{s(g)k_x}(y,b) & = \lambda_{k_{\lambda_g(x)}s(g)}(y,b) \\
        g \cdot (y, b + \Phi(x,y)) & =  \Phi(\lambda_g(x),\lambda_g(y)) + g \cdot (y,b) \quad (\textnormal{because } g \cdot (y,b) \in \lambda_g(y)) \\
        g \cdot (y,b) + g \cdot \Phi(x,y) & = g \cdot (y,b) + \Phi(\lambda_g(x),\lambda_g(y)).
    \end{align*}
    Cancelling $g \cdot (y,b)$ we have proven that $g \cdot \Phi(x,y) = \Phi(\lambda_g(x),\lambda_g(y))$.
\end{proof}

Given a cycle set of the form $Z = X \otimes_{\Phi}A$, we write $Z_0 = \{ (x,0): x \in X \} \subseteq Z$. A careful investigation of the construction shows that we can choose any $\G(Y)_{\pi}$-invariant section of $f: Y \twoheadrightarrow X$ to be the set $Z_0$ under our coordinatization procedure:

\begin{cor} \label{cor:sections_can_be_chosen_as_zero_sections}
    Let $f: Y \twoheadrightarrow X$ be a coprime extension of indecomposable cycle sets. Let $\pi = \pi(|\G(X)|)$ and let $\mathfrak{O} \subseteq Y$ be a $\G(Y)_{\pi}$-orbit such that the restriction $f|_{\mathfrak{O}}: \mathfrak{O} \to X$ is bijective. Then there is a twisted extension $\pr: Y' = X \otimes_{\Phi} A \twoheadrightarrow X$, together with an isomorphism $g: Y' \overset{\sim}{\to} Y$ such that $f \circ g = \pr$ and $g (Y^{\prime}_0) = \mathfrak{O}$.
\end{cor}

Recall that a cycle set $X$ is called \emph{uniconnected} if $X$ is indecomposable and $\Pi_1(X,x_0) = 1$ which is the same as saying that $\G(X)$ acts regularly on $X$. In this case, each $X$-graded $\G(X)$-module is isomorphic to a permutation module, and \cref{thm:coprime_extensions_are_twisted_extensions} specializes to:

\begin{cor} \label{cor:coprime_extensions_of_a_uniconnected_cycle_set}
    Let $f: Y \to X$ be a coprime extension of indecomposable cycle sets where $X$ is uniconnected. Then $f$ is equivalent to a parallel extension $\pr: X \times_{\Gamma} A \twoheadrightarrow X$.
\end{cor}

\section{Permutation groups of coprime extensions of cycle sets} \label{sec:permutation_groups}

In this section, we will give a simple description of $\G(Y)$ for $Y = X \otimes_{\Phi} A$ if $X$ and $A$ are chosen in a suitable way.

Given a map $\Phi: X \times X \to A$ satisfying \cref{eq:phixy_in_ay,eq:phi_equivariance}, we define the elements $\Phi_x = \sum_{y \in X} \Phi(x,y) \in A$ ($x \in A$). Using these elements, we can also express \cref{eq:phi_equivariance} as
\begin{equation}
    g \cdot \Phi_x = \Phi_{\lambda_g(x)}.
\end{equation}

The following result describes the permutation brace of a twisted extension that results in a coprime extension of indecomposable cycle sets. Note that for an $X$-graded $G$-module $A$, we identify $A$ with the subgroup of $\Sym_{A^{\sca}}$ that acts by $f \cdot (x,b) = (x,b+f_x)$ for $f \in A$, $(x,a) \in A^{\sca}$.

\begin{thm} \label{pro:generating_kernel_of_Gpr}
    Let $X$ be indecomposable, and let $\pi = \pi(|\G(X)|)$. Consider a twisted extension $\pr: Y = X \otimes_{\Phi} A \twoheadrightarrow X$ where $\gcd(|\G(X)|,|A_x|) = 1$ for $x \in X$, where $Y$ is indecomposable, then there is a semidirect product decomposition
    \[
    \G(Y) = \G(Y)_{\pi} \ltimes K
    \]
    where $\G(Y)_{\pi} = \{ (x,a) \mapsto (\lambda_g(x), g \cdot a): g \in \G(X)\} \cong \G(X)$ and $K = \langle \Phi_x : x \in X \rangle \leq A$ is a trivial brace such that the $\lambda$-action of $\G(Y)_{\pi}$ on $K$ coincides with $\G(X)$-module structure on $K$ in the following sense: for $f \in K$ acting by $\lambda_f(x,a) = (x,a+f_x)$ and $g \in \G(X)$, we have $\lambda_{s(g)}(f) = g \cdot f$, where $s: \G(X) \to \G(Y)$ denotes the unique section of $\G(\pr):\G(Y) \twoheadrightarrow \G(X)$ with $s(\G(X)) = \G(Y)_{\pi}$.
\end{thm}

Before proving the result, note that this statement is not completely trivial! It is tempting to reverse the construction from the proof of \cref{thm:coprime_extensions_are_twisted_extensions} and decompose the cycle set permutation $(y,b) \mapsto \lambda_x^{-1} \cdot (y, b + \Phi(x,y))$ into components $(y,b) \mapsto (y,b + \Phi(x,y))$ and $(y,b) \mapsto \lambda_x^{-1} \cdot (y,b)$, the latter coming from the action of the $\pi$-subgroup in the proof. However, it is not clear at all that for an arbitrary extension $\pr: X \otimes_{\Phi} A \twoheadrightarrow X$, the construction in the proof will reproduce the parameter $\Phi$ as there might be many parameters giving the same solution. We need some lemmata before proving the result.

\begin{lem} \label{lem:GY_pi_fixes_a_section}
    Let $X$ be an indecomposable cycle set and let $A$ be an $X$-graded $\G(X)$-module with $\gcd(|\G(X)|,|A_x|)=1$ for all $x$ and suppose that $\pr: Y = X \otimes_{\Phi} A \twoheadrightarrow X$ a twisted extension that is not necessarily indecomposable. Let $\pi = \pi(|\G(X)|)$. Then there is a (set-theoretic) section $s: X \hookrightarrow Y$ with $\pr \circ s = \id_X$ such that $\G(Y)_{\pi}$ fixes $s(X)$.
\end{lem}

\begin{proof}
    We construct an affine action of $\G(Y)$ on $A$ as follows: for $g \in \G(Y)$, write $\Bar{g} = \G(\pr)(g)$. It is easy to check that $g \cdot (a_x)_{x \in X} = \lambda_g(a_{\Bar{g}^{-1}(x)})_{x \in X}$ defines an action of $\G(Y)$ on $A$. Writing the action of a $g \in \G(Y)$ on $Y$ as $\lambda_g(x,a_x) = (\lambda_{\Bar{g}}(x), \Bar{g} \cdot (a_x + b_{g,x}))$, this action can be described by the formula:
    \[
    g \cdot (a_x)_{x \in X} =  (\Bar{g} \cdot (a_{\lambda_{\Bar{g}}^{-1}(x)} + b_{g,\lambda_{\Bar{g}}^{-1}(x)}))_{x \in X}.
    \]
    It is now immediate that the action represents $\G(Y)$ as a subgroup of $\Hol(A)$. As $\gcd(|\G(Y)_{\pi}|,|A|) = 1$, we can again use \cref{lem:coprime_affine_actions_have_fixed_points} to conclude that $\G(Y)_{\pi}$ has a fixed point on $A$. Let $(a_x)_{x \in X}$ be such a fixed point, then $\lambda_g(a_x) = a_{\lambda_{\Bar{g}}(x)}$ for all $x \in X$, $g \in \G(Y)_{\pi}$, so $\G(Y)_{\pi}$ fixes the set $\mathcal{O} = \{(x,a_x): x \in X \}$, which is the image of the set-theoretic section $s: X \hookrightarrow Y$; $x \mapsto (x,a_x)$.
\end{proof}

Recall that for a twisted extension $Y = X \otimes_{\Phi}A$, we define $Y_0 = \{ (x,0): x \in X \} \subseteq Y$.

\begin{lem} \label{lem:embeddings_maps_Z0_to_orbits}
    Let $X$ be an indecomposable cycle set, take a point $x_0 \in X$, and let $A$ be an $X$-graded $\G(X)$-module with $\gcd(|\G(X)|,|A_x|) = 1$. Let $Y = X \otimes_{\Phi}A$ be a twisted extension where $Y$ is not necessarily indecomposable. Furthermore, let $B \leq A$ be an $X$-graded submodule of $A$ and let $Z = X \otimes_{\Gamma}B$ be another extension where $Z$ is not necessarily indecomposable. Let $f: Z \hookrightarrow Y$ be an embedding of cycle sets such that $f(x,a) = (x,a+\delta_x)$ for a fixed family of $\delta_x \in A_x$ ($x \in X$) and suppose that $\mathcal{O}_{\G(Y)_{\pi}}(x_0,0) = Y_0$. Then $f(Z_0)$ is a $\G(Y)_{\pi}$-orbit in $Y$ and $\Gamma = \Phi$.
\end{lem}

\begin{proof}
    The homomorphism property reads
    \begin{align*}
        f((x,a) \ast (y,b)) & = f(x,a) \ast f(y,b) \\
        (x \ast y,\lambda_x^{-1} \cdot (b + \Gamma(x,y)) + \delta_{x \ast y}) & = (x \ast y, \lambda_x^{-1} \cdot (b + \delta_y + \Phi(x,y))) \\
        \Rightarrow \lambda_x^{-1} \cdot (\Gamma(x,y) + \delta_{x \ast y}) & = \lambda_x^{-1} \cdot (\delta_y + \Phi(x,y)) \\
        \lambda_x \cdot \delta_{x \ast y} - \delta_y & = \Phi(x,y) - \Gamma(x,y). \tag{$\dagger$}
    \end{align*}
    As the right side is $\G(X)$-equivariant, the left side is, either, meaning that
    \begin{equation}\label{eq:equivariance_of_delta_term}
        \lambda_x \cdot \delta_{x \ast y} - \delta_y = \lambda_g^{-1} \cdot (\delta_{\lambda_g(x) \ast \lambda_g(y)} - \delta_{\lambda_g(y)}).
    \end{equation}
    Taking $\lambda_g = \lambda_x^{-1}$ in \cref{eq:equivariance_of_delta_term}, we obtain:
    \begin{align*}
    \lambda_x \cdot \delta_{x \ast y} - \delta_y & = \lambda_x \cdot (\lambda_{\lambda_x^{-1}(x)} \cdot \delta_{\lambda_x^{-1}(x) \ast \lambda_x^{-1}(y)}  - \delta_{\lambda_x^{-1}(y)}) \\
    & = \lambda_x \cdot (\lambda_{\Sq(x)} \cdot \delta_{(x \ast x) \ast (x \ast y)} - \delta_{x \ast y}) \\
    & = \lambda_x \lambda_{\Sq(x)} \cdot \delta_{(x \ast x) \ast (x \ast y)} - \lambda_x \cdot \delta_{x \ast y} \\
    & = \lambda_{2x} \cdot \delta_{2x \ast y} - \lambda_x \cdot \delta_{x \ast y},
    \end{align*}
    where we write $kx \ast y = \lambda_{kx}^{-1}$ with $\lambda_{kx} = k \cdot \lambda_x$. By induction over $k$, using $g + h = g \lambda_g^{-1}(h)$ and the fact that $\Sq(x) = \lambda_x^{-1}(x)$, one shows that
    \[
    \lambda_{kx} = \lambda_x \lambda_{\Sq(x)} \ldots \lambda_{\Sq^{k-1}(x)}.
    \]
    An iterated application of the argument above proves for $k \geq 0$ that:
    \begin{equation} \label{eq:iteration_for_delta}
    \lambda_x \cdot \delta_{x \ast y} - \delta_y = \lambda_{(k+1)x} \cdot \delta_{(k+1)x \ast y} - \lambda_{kx} \cdot \delta_{kx \ast y}.
    \end{equation}
    Let $d$ be the exponent of $\G(X)^+$, which is also the smallest positive integer such that $k \cdot \lambda_x = 0$ for all $x \in X$. With this $d$, we get, using \cref{eq:iteration_for_delta}:
    \begin{align*}
    d \cdot (\lambda_x \cdot \delta_{x \ast y} - \delta_y) & = \sum_{k=0}^{d-1} \lambda_{(k+1)x} \cdot \delta_{(k+1)x \ast y} - \lambda_{kx} \cdot \delta_{kx \ast y} \\
    & =\lambda_{dx} \cdot \delta_{dx \ast y} - \lambda_{0x} \cdot \delta_{0x \ast y} \\
    & = \delta_y - \delta_y = 0.
    \end{align*}
    Now $d$ divides $|\G(X)|$ (see also \cite[Proposition 2.7]{Feingesicht}), therefore $\gcd(d,|A|) = 0$ from which we infer that $\lambda_x \cdot \delta_{x \ast y} - \delta_y = 0$ for all $x,y \in X$, resp. $\delta_{\lambda_x^{-1}(y)} = \lambda_x^{-1} \cdot \delta_y$. As the $\lambda_x$ ($x \in X$) generate $\G(X)$, it follows that $\delta_{\lambda_g(y)} = g \cdot \delta_y$ holds for all $g \in \G(X)$, $y \in X$, therefore the set
    \[
    \{ (x, \delta_x) : x \in X \} \subseteq X \otimes_{\Phi} A
    \]
    is a $\G(Y)_{\pi}$-orbit in $Y$. The observation that $f(x,0) = (x,\delta_x)$ for $(x,0) \in Z_0$ concludes the proof that $f(Z_0)$ is a $\G(Y)_{\pi}$-orbit in $Y$.

    Furthermore, $\lambda_x \cdot \delta_{x \ast y} = \delta_{\lambda_x(x \ast y)} = \delta_y$, so $(\dagger)$ implies that $\Phi(x,y) - \Gamma(x,y) = 0$, i.e. $\Phi = \Gamma$.
\end{proof}

\begin{proof}[Proof of \cref{pro:generating_kernel_of_Gpr}]
    By \cref{thm:coprime_extensions_are_twisted_extensions}, $\pr$ is isomorphic to $\pr^{\prime}: Y^{\prime} = X \otimes_{\Phi^{\prime}}A^{\prime} \twoheadrightarrow X$ for some $\G(X)$-equivariant map $\Phi^{\prime}: X \times X \to A^{\prime}$. Note that the structure of the underlying $X$-graded $\G(X)$-module is determined, up to isomorphism, by the $\lambda$-action of $\G(Y)_{\pi}$ on the fibers $\mathcal{F}_x = A_x$ resp $A_x^{\prime}$ (see the proof of \cref{thm:coprime_extensions_are_twisted_extensions}), so we may assume that $A^{\prime} = A$.
    
    It is sufficient to show that $\G(Y)_{\pi}$ consists of the permutations $(x,a) \mapsto (\lambda_g(x),g \cdot a)$ with $g \in \G(X)$ as then it follows that $K$ is generated by the maps $(y,b) \mapsto (y,b+ \Phi(x,y))$ for $x \in X$.
    
    This is done by showing that $Y_0 \subseteq Y$ is a $\G(Y)_{\pi}$-orbit, as then, the $\G(X)$-module structure on $A$ forces the elements of $\G(Y)_{\pi}$ to act as $(x,a) \mapsto (\lambda_g(x),g \cdot a)$.
    
    Note first that \cref{lem:GY_pi_fixes_a_section} tells us there is a $\G(Y)_{\pi}$-orbit $\mathfrak{O} \subseteq Y$ of the form $\mathfrak{O} = \{ (x,a_x): x \in X\}$, so by \cref{cor:sections_can_be_chosen_as_zero_sections} to \cref{thm:coprime_extensions_are_twisted_extensions}, there is a twisted extension $\pr': Y' = X \otimes_{\Gamma} A \twoheadrightarrow X$ and an equivalence $g: Y \to Y'$ between the extensions $\pr, \pr^{\prime}$ such that $g(x,a_x) = (x,0)$ for all $x \in X$. As $g$ respects the regular action of $A$ on $A^{\mathrm{sc}}$, $g$ must be of the form $g(x,a) = (x,a-a_x)$. Now \cref{lem:embeddings_maps_Z0_to_orbits} tells us that $\Gamma = \Phi$. On the other hand, the same lemma says that $f(Y_0)$ is a $\G(Y')_{\pi}$-orbit in $Y'$. By the isomorphism $\G(g): \G(Y) \overset{\sim}{\to} \G(Y')$, it follows that $Y_0$ is a $\G(Y)_{\pi}$-orbit, which is what we wanted to show!

    As $K \leq \Soc(\G(Y))$, it is clear that $K$ is trivial. From \cref{pro:lambda_on_socle_is_conjugation} it follows that $\lambda_{s(g)}(f) = {}^{s(g)}f$ for $g \in \G(X)$, $f \in K$. By our proof, $\lambda_{s(g)}(x,a) = (\lambda_g(x), g \cdot a)$ holds for $(x,a) \in Y$. Therefore, we calculate
    \begin{align*}
    \lambda_{\lambda_{s(g)}(f)}(x,a) & = \lambda_{s(g)} ( \lambda_f (\lambda_g^{-1}(x),g^{-1} \cdot a) ) \\
    & = \lambda_{s(g)} (\lambda_g^{-1}(x), g^{-1} \cdot a + f(\lambda_g^{-1}(x))) \\
    & = (x, a + g \cdot f(\lambda_g^{-1}(x))) \\
    & = \lambda_{g \cdot f} (x,a).
    \end{align*}
\end{proof}

We now reformulate this result in terms of parallel extensions. For a map $\Gamma: X \times X \to B$, where $B$ is an (ungraded) group, we define the element $\Gamma_x \in B^X$ by $\Gamma_x(y) = \Gamma(x,y)$.

\begin{cor} \label{cor:generating_kernel_of_parallel_extension}
    Let $X$ be indecomposable, and let $\pi = \pi(|\G(X)|)$. Consider a parallel extension $\pr: Y = X \times_{\Gamma} B \twoheadrightarrow X$ where $\gcd(|\G(X)|,|B|) = 1$, and where $Y$ is indecomposable, then there is a semidirect product decomposition
    \[
    \G(Y) = \G(Y)_{\pi} \ltimes K
    \]
    where $\G(Y)_{\pi} = \{ (x,a) \mapsto (\lambda_g(x), a): g \in \G(X)\} \cong \G(X)$ and $K = \langle \Gamma_x : x \in X \rangle \leq B^X$ is a trivial brace. Furthermore for $f \in K$ acting by $\lambda_f(x,a) = (x,a+f(x))$ and $g \in \G(X)$, we have $\lambda_{s(g)}(f) = {}^gf$, where $s: \G(X) \to \G(Y)$ is the unique section of $\G(\pr):\G(Y) \twoheadrightarrow \G(X)$ with $s(\G(X)) = \G(Y)_{\pi}$ and where ${}^gf(x) = f(\lambda_g^{-1}(x))$.
\end{cor}

We can now characterize the cases in which $X \otimes_{\Phi} A$ is indecomposable:

\begin{thm} \label{thm:indecomposability_criterion}
    Suppose $X$ is an indecomposable cycle set and $A$ is an $X$-graded $\G(X)$-module with $\gcd(|\G(X)|,|A_x|) = 1$. Then a twisted extension $X \otimes_{\Phi}A$ is indecomposable if and only if for any $x_0 \in X$, we have
    \begin{equation} \label{eq:Ax0_generated_by_phixx0}
    A_{x_0} = \left\langle \Phi(x,x_0): x \in X \right\rangle.
    \end{equation}
\end{thm}

\begin{proof}
    Assume first that $Y = X \otimes_{\Phi}A$ is indecomposable, then \cref{pro:generating_kernel_of_Gpr} implies that $\G(Y) = \G(X)_{\pi} \ltimes K$ where $K = \left\langle \Phi_x : x \in X \right\rangle$. From the proof of \cref{pro:generating_kernel_of_Gpr} it follows that 
    \[
    \mathcal{O}_{\G(X)_{\pi}}(x_0,0)\cap (\{ x_0\} \times A_{x_0}) = \{(x_0,0)\}.
    \]
    Therefore, $K$ acts transitively on $\{ x_0\} \times A_{x_0}$ which is the same as saying that $\left\langle \Phi_x(x_0) : x \in X \right\rangle = A_{x_0}$.

    Suppose now that \cref{eq:Ax0_generated_by_phixx0} is satisfied. By \cref{lem:GY_pi_fixes_a_section}, $Y$ contains a $\G(Y)_{\pi}$-orbit of the form $\mathfrak{O} = \{ (x,a_x): x \in X \} \subseteq Y$. Fixing this orbit and the elements $a_x$, we let $Y^{\prime} = \mathcal{O}_{\G(Y)}(x_0,a_{x_0})$ for some $x_0 \in X$ and consider the map $\pr' = \pr|_{Y'}: Y' \to X$. As $X$ is indecomposable and $\G(\pr)$ is surjective, it follows that $\pr'$ is indeed surjective. As the permutations $\sigma_{(x,a)}$ for $(x,a) \in X \otimes_{\Phi}A$ only depend on the $X$-coordinate, it follows that $\G(Y') = \G(Y)$, so $Y' = \mathcal{O}_{\G(Y')}(x_0,a_{x_0})$ which implies that $Y'$ is indecomposable.

    Furthermore, $\pr': Y' \twoheadrightarrow X$ is a coprime extension: let $\pi = \pi(|\G(X)|)$ and let $\pi'$ be the set of primes not in $\pi$, then
    \[
    \pi \left( \frac{|Y'|}{|X|} \right) =  \pi \left( \frac{|Y'|}{|\mathfrak{O}|} \right) = \pi \left( \frac{|\G(Y)|}{|\G(Y)_{\pi}|} \cdot \frac{|\Pi_1(Y,y_0) \cap \G(Y)_{\pi}|}{|\Pi_1(Y,y_0)|} \right) \subseteq \pi \left( \frac{|\G(Y)|}{|\G(Y)_{\pi}|}\right) = \pi'.
    \]
    It follows that there is an isomorphism $X \otimes_{\Gamma} B \cong Y'$ for some submodule $B \leq A$ as $Y' \subseteq X \otimes_{\Phi} A$; the graded components $B_x$ of $B$ are the regular subgroups induced by $\ker(\G(\pr')) \leq A$ on the fibers $\pr'^{-1}(x)$. Furthermore, $\mathfrak{O} = \{(x,a_x): x \in X \}$ is a $\G(Y')_{\pi}$-orbit, so by \cref{cor:sections_can_be_chosen_as_zero_sections}, we can identify the embedding $Y' \hookrightarrow Y$ with an embedding $X \otimes_{\Gamma} B \hookrightarrow X \otimes_{\Phi} A$; $(x,a) \mapsto (x,a + a_x)$. \cref{lem:embeddings_maps_Z0_to_orbits} now implies $\Gamma = \Phi$. As $Y'$ is indecomposable, we conclude from the first part of the theorem that
    \[
    B_{x_0} = \left\langle \Phi(x,x_0): x \in X \right\rangle = A_{x_0}
    \]
    for each $x_0 \in X$. Therefore, $B = A$ which implies that $Y = Y'$ from which we infer that $Y$ is indecomposable.
\end{proof}

The corresponding corollary for parallel extensions is stated as follows:

\begin{cor} \label{cor:indecomposability_criterion}
    Suppose that $X$ is an indecomposable cycle set and $B$ is a group that satisfies $\gcd(|\G(X)|,|B|) = 1$. Then a parallel extension $X \times_{\Gamma}B$ is indecomposable if and only if 
    \begin{equation} \label{eq:B_generated_by_gammax}
    B = \left\langle \Gamma(x,y): x \in X \right\rangle.
    \end{equation}
\end{cor}

\section{Indecomposable cycle sets of size \texorpdfstring{$pqr$}{pqr}} \label{sec:size_pqr}

The goal of this section is to give a full description of indecomposable cycle sets of size $pqr$. We will rely on the following result of Cedó and Okniński

\begin{thm}[Cedó, Okniński] \label{thm:cedo_okninski}
    Let $X$ be an indecomposable cycle set of squarefree size. Then $\mpl(X) < \infty$. Furthermore, $\pi(|\G(X)|) = \pi(|X|)$.
\end{thm}

\begin{proof}
    See \cite[Theorem 4.1, Theorem 4.5]{CO_SquarefreeIndecomposable}.
\end{proof}

\begin{cor} \label{cor:squarefree_cyclesets_are_coprime_extensions}
    Let $X$ be an indecomposable cycle set of squarefree size, then the homomorphism $\ret: X \twoheadrightarrow X^{(1)}$ is a coprime extension.
\end{cor}

\begin{proof}
    Let $|X|$ be squarefree then $\pi(|\G(X^{(1)})|) = \pi(|X^{(1)}|)$, by \cref{thm:cedo_okninski}. Furthermore, the squarefreeness of $|X|$ implies
    \[
    \emptyset = \pi\left( \frac{|X|}{|X^{(1)}|} \right) \cap \pi(|X^{(1)}|)  = \pi\left( \frac{|X|}{|X^{(1)}|} \right) \cap \pi(|\G(X^{(1)})|),
    \]
    so $\gcd\left(\frac{|X|}{|X^{(1)}|}, |\G(X)| \right) = 1$, that is, $\ret: X \twoheadrightarrow X^{(1)}$ is a coprime extension.
\end{proof}

Together with the apparatus developed in \cref{sec:coprime_extensions}, \cref{cor:squarefree_cyclesets_are_coprime_extensions} allows us to implement the following strategy to construct all indecomposable cycle sets of squarefree size $n$:

\begin{enumerate}
    \item Suppose that all indecomposable cycle sets of sizes $m|n$, $m \neq n$ have already been constructed.
    \item For given $m$ with $n = m \cdot k$, consider the cycle sets $X^{\prime}$ with $|X^{\prime}| = m$.
    \item Construct all $X^{\prime}$-graded $\G(X^{\prime})$-modules $A$ with $|A_{x_0}| = k$.
    \item Find all $\G(X^{\prime})$-equivariant mappings $\Phi: X^{\prime} \times X^{\prime} \to A$ that satisfy the indecomposability criterion in \cref{thm:indecomposability_criterion}.
\end{enumerate}

We start by repeating a well-known case:

\begin{pro} \label{pro:mp1_indecomposable_implies_cyclic}
    If $X$ is an indecomposable cycle set with $\mpl(X) = 1$ and $|X| = n$, then $X \cong \Z_n$ with $x \ast y = y+1$. In particular, $X$ is uniconnected.
\end{pro}

\begin{proof}
    As $\mpl(X) = 1$, we can write $x \ast y = \sigma(y)$ for some permutation $\sigma \in \Sym_X$. Clearly $\G(X) = \left\langle \sigma \right\rangle$ which is transitive if and only if $\sigma$ is an $n$-cycle. Therefore, we can choose an $x_0 \in X$ and construct the desired isomorphism as
    \[
    f: \Z_n \to X ; \quad k \mapsto \sigma^k(x_0).
    \]
    Uniconnectivity is obvious.
\end{proof}

The following result concerns cycle sets of squarefree size and multipermutation level $2$ which have been constructed in greater generality by Jedli\v{c}ka and Pilitowska \cite{jedlicka_pilitowska}. We nevertheless show how to obtain these cycle sets by the method developed in this article; furthermore, it gives us the opportunity to express those cycle sets in a way that is handy for further computations concerning cycle sets of multipermutation level $3$.

\begin{pro} \label{thm:classification_pq}
    Let $X$ be an indecomposable cycle set of squarefree size $n = m \cdot k$. If $\mpl(X) = 2$ and $|X^{(1)}| = m$, then there is an isomorphism with a parallel extension $X \cong \Z_m \times_{\Gamma} \Z_k$ where $\Z_m$ carries the cycle set operation $x \ast y = y + 1$.
    
    Putting $\Gamma^0(x)= \Gamma(0,x)$, the cycle set operation of $\Z_m \times_{\Gamma} \Z_k$ can be expressed by
    \begin{equation} \label{eq:classification_mp2}
        (x,a) \ast (y,b) = (y+1, b + \Gamma^0(y-x)).
    \end{equation}
    Vice versa, each map $\Gamma^0: \Z_m \to \Z_k$ provides an indecomposable cycle set with $\mpl(X) = 2$ and $|X^{(1)}| = m$ if $\Z_k = \left\langle \Gamma^0(x): x \in \Z_m \right\rangle$ and if for all $0 \neq d \in \Z_m$, there is an $x \in \Z_m$ with $\Gamma^0(x+d) \neq \Gamma^0(x)$.
\end{pro}

\begin{proof}
    If $X$ is indecomposable and $\mpl(X) =2$, then $X^{(1)}$ is indecomposable with $\mpl(X^{(1)})=1$. By \cref{pro:mp1_indecomposable_implies_cyclic}, $X^{(1)} \cong \Z_m$, with the cycle set operation $x \ast y = y + 1$. Identifying $X^{(1)} = \Z_m$, the retraction map $X \twoheadrightarrow \Z_m$ is a coprime extension, by \cref{cor:squarefree_cyclesets_are_coprime_extensions}. As $\Z_m$ is uniconnected, \cref{cor:coprime_extensions_of_a_uniconnected_cycle_set} implies that $X$ is a parallel extension of $\Z_m$ by an abelian group of order $k$ which must be isomorphic to $\Z_k$. It follows that $X \cong \Z_m \times_{\Gamma}\Z_k$ for a $\G(X^{(1)})$-invariant map $\Gamma: \Z_m \times \Z_m \to \Z_k$.

    Putting $\Gamma^0(x) = \Gamma(0,x)$, we can use \cref{eq:cycle_operation_parallel_extension} to express the operation on $\Z_m \times_{\Gamma}\Z_k$ as
    \[
    (x,a) \ast (y,b) = (y+1,b + \Gamma^0(\lambda_g^{-1}(y)))
    \]
    with a fixed map $\Gamma^0: \Z_m \to \Z_k$ where $g \in \G(X^{(1)})$ is chosen such that $\lambda_g(0)=x$. Obviously, $g$ must satisfy $\lambda_g(z) = z + x$, therefore $\lambda_g^{-1}(y) = y-x$ which leads to the description by \cref{eq:classification_mp2}.

    By \cref{cor:indecomposability_criterion}, given $\Gamma^0: \Z_m \to \Z_k$, the cycle set on $\Z_m \times \Z_k$ defined by \cref{eq:classification_mp2} is indecomposable if and only if
    \[
    \Z_k = \left\langle \Gamma(x,y) : x,y \in \Z_m \right\rangle = \left\langle \Gamma^0(y-x) : x,y \in \Z_m \right\rangle = \left\langle \Gamma^0(x) : x \in \Z_m \right\rangle.
    \]
    If the projection $\Z_m \times \Z_k \to \Z_m$; $(x,a) \mapsto x$ is a retraction then for any $0 \neq d \in \Z_m$ we have $\sigma_{(0,0)} \neq \sigma_{(-d,0)}$, which implies that there is an $x \in \Z_m$ with $\Gamma^0(x-0) \neq \Gamma^0(x-(-d))$, i.e. $\Gamma^0(x) \neq \Gamma^0(x+d)$. If, on the other hand, $\sigma_{(x,0)} = \sigma_{(y,0)}$ for distinct $x,y \in \Z_m$, then for all $z \in \Z_m$, we have $\Gamma^0(z-x) = \Gamma^0(z-y)$. Substituting $z' = z-x$ and putting $d = x-y$, we see that for all $z' \in \Z_m$, we have $\Gamma^0(z') = \Gamma^0(z'+d)$, while $0 \neq d \in \Z_m$.
\end{proof}

We now want to tackle the classification of indecomposable cycle sets $X$ of size $pqr$ where $p,q,r$ are distinct primes and where $\mpl(X) = 3$. Without restriction, we can assume that $|X^{(1)}| = pq$ and $|X^{(2)}| = p$. Using \cref{thm:classification_pq}, we can represent $X^{(1)} = \Z_p \times_{\Gamma} \Z_q$. As $q$ is prime, \cref{thm:classification_pq} implies that $\Gamma: \Z_p \times \Z_p \to \Z_q$ indeed defines an indecomposable cycle set if and only if $\Gamma^0 = \Gamma(0,-)$ is nonzero and, by the primality of $p$, the multipermutation level of this cycle set is $2$ if and only if $\Gamma^0$ is nonconstant which is equivalent to $\Gamma$ being nonconstant.

We first want to understand the stabilizers in $X^{(1)}$ better. To this end, we prove the following result:

\begin{pro} \label{pro:orbits_of_stabilizer_pq}
    Let $X = \Z_p \times_{\Gamma} \Z_q$ where $\Gamma$ is nonconstant. Let $x_0 = (0,0)$. Then the orbits of $\Pi_1(X,x_0)$ on $X$ are given by:
    \begin{enumerate}
        \item the singletons $\{ (x,a) \}$ if $\Gamma^0$ is \emph{proportional to a one-dimensional character}, that is, $\Gamma^0 = \alpha \cdot \chi$ for some $\alpha \in \Z_q^{\ast}$ and some $\chi: \Z_p \to \Z_q^{\times}$ with $\chi(x+y) = \chi(x) \chi(y)$.
        \item else, the singletons $\{ (0,a)\}$ ($a \in \Z_q$) and the blocks $\{ x \} \times \Z_q$ ($x \in \Z_p^{\ast}$).
    \end{enumerate}
\end{pro}

Our approach is to exploit the functional description of $K = \ker(\G(X^{(1)}) \twoheadrightarrow \G(X^{(2)}))$ from \cref{pro:generating_kernel_of_Gpr} to connect stabilization to the vanishing of functions.

First of all, we present a very broad framework for this idea: let $M$ be a set and $N \subseteq M$. Let $A$ be an abelian group and consider a subgroup $H \leq A^M$. We want to measure to what extent the vanishing of a function $f \in H$ on $N$ forces $f$ to vanish on other points too. To this end, we define the set
\[
\mathcal{Z}_{H,N} = \{x \in M \ : \ \quad \forall f \in H: f|_N \equiv 0 \Rightarrow f(x) = 0  \}.
\]
Obviously $N \subseteq \mathcal{Z}_{H,N}$.

Our motivation to introduce this concept comes from the fact that the orbits of the stabilizer can be described in terms of $\mathcal{Z}_{H,N}$ if we consider extensions of uniconnected cycle sets. Recall that for $\Gamma: X \times X \to B$ where $B$ is an (ungraded) group, we define $\Gamma_x \in B^X$ by $\Gamma_x(y) = \Gamma(x,y)$.

\begin{lem} \label{lem:zeros_of_stabilizer}
    Let $X$ be a uniconnected cycle set, $q$ a prime, and suppose that $\pr: Y = X \times_{\Gamma} \Z_q \twoheadrightarrow X$ is a coprime, parallel extension of indecomposable cycle sets.
    
    Let $y_0 = (x_0,0) \in Y$ and put
    \[
    H = \genrel{\Gamma_x}{x \in X} \leq \Z_q^X.
    \]
    Then the orbit of $(x,a) \in Y$ under $\Pi_1(Y,y_0)$ is
    \[
    \mathcal{O}_{\Pi_1(Y,y_0)}(x,a) = \begin{cases}
        \{(x,a) \} & x \in \mathcal{Z}_{H,\{x_0\}} \\
        \{ x \} \times \Z_q & x \not\in \mathcal{Z}_{H,\{x_0\}}.
    \end{cases}
    \]
\end{lem}

\begin{proof}
    The induced action of $\Pi_1(Y,y_0)$ on the quotient $X$ fixes $x_0$, thus it fixes the entirety of $X$, as $X$ is uniconnected. It follows that
    \[
    \Pi_1(Y,y_0) \subseteq \ker(\G(Y) \to \G(X)) \leq \Soc(G(Y)) \overset{\textnormal{\cref{cor:generating_kernel_of_parallel_extension}}}{=} \genrel{\Gamma_x}{x \in X} = H.
    \]
    As an element $f \in H$ acts as $f \cdot (x,a) = (x,a+ f(x))$, it is now obvious that $\Pi_1(Y,y_0)$ can be identified with $H_0 = \{ f \in H: f(0) = 0 \}$.
    
    Furthermore, $(x,a)$ is fixed under the action of $H_0$ if and only if $f(x) = 0$ for all $f \in H_0$, i.e. $x \in \mathcal{Z}_{H,\{x_0\}}$. On the other hand, if $x \not\in \mathcal{Z}_{H,\{x_0\}}$, we see similarly that $\mathcal{O}_{H_0}(x,a) \supseteq \{ x \} \times \Z_p$. As $H_0$ fixes all sets $\{ x\} \times \Z_q$ setwise, we conclude that $\mathcal{O}_{H_0}(x,a) = \{ x \} \times \Z_q$.
\end{proof}

We can now state our lemma. For this, we denote by $\mathbb{F}_q$ the field with $q$ elements.

\begin{lem} \label{lem:zeros_of_invariant_subspaces}
    Let $p,q$ be different primes and let $U \leq \mathbb{F}_q^{\Z_p}$ be a $\Z_p$-invariant subspace with respect to the regular representation. Then
    \[
    \mathcal{Z}_{U,\{0\}} = \begin{cases}
        \Z_p & \dim_{\mathbb{F}_q}U \leq 1 \\
        \{ 0 \} & \dim_{\mathbb{F}_q}U \geq 2.
    \end{cases}
    \]
\end{lem}

\begin{proof}
    It is immediate that $0 \in \mathcal{Z}_{U,\{ 0 \}}$. Suppose that there is a $d \in \Z_p$ such that $0 \neq d \in \mathcal{Z}_{U,\{ 0 \}}$, then for $f \in U$, the implication $f(0) = 0 \Rightarrow f(d) = 0$ is valid by definition. As $U$ is $\Z_p$-invariant, the implication $f(x) = 0 \Rightarrow f(x +d) = 0$ holds for $f \in U$ as well. This shows that $\mathcal{Z}_{U,\{ 0 \}}$ is invariant under the mapping $x \mapsto x + d$. Therefore, $\mathcal{Z}_{U,\{ 0 \}} = \Z_p$. In this case, we argue as follows that $\dim_{\mathbb{F}_q}U \leq 1$: if $\dim_{\mathbb{F}_q}U \geq 1$, there is an $x_0 \in \Z_p$ such that $f(x_0) \neq 0$ for at least one $f \in U$. By shift-invariance, we can chose $x_0$; therefore the $\mathbb{F}_q$-linear map $\varphi: U \to \mathbb{F}_q$; $f \mapsto f(0)$ is surjective with $\dim_{\mathbb{F}_q} \ker (\varphi) = \dim_{\mathbb{F}_q}U - 1 $. But for all $f \in \ker(\varphi)$, we have $f(0)$, so as $\mathcal{Z}_{U,\{ 0 \}} = \Z_p$ it follows that $f(x) = 0$ for all $x \in \Z_p$. Therefore, $\ker \varphi = \{ 0 \}$ which implies $\dim_{\mathbb{F}_q}U = 1$. Therefore, we have shown that $\mathcal{Z}_{U,\{ 0 \}} = \{ 0 \}$ if $\dim_{\mathbb{F}_q}U \geq 2$.

    If $\dim_{\mathbb{F}_q}U = 0$, it is obvious that $\mathcal{Z}_{U,\{ 0 \}} = \Z_p$. On the other hand, if $\dim_{\mathbb{F}_q}U = 1$, then we can pick an $f \in U$ with $f(0) = 1$. By $\Z_p$-invariance, $\dim_{\mathbb{F}_q}U = 1$ implies there is a $\xi \in \F_q$ such that $f(x+1) = \xi \cdot f(x)$ for all $x$. Therefore, $f(x) = \xi^x$ and $U = \{ x \mapsto \alpha \cdot \xi^x : \alpha \in \F_q \}$. As $\xi^x \neq 0$ for all $x$ it now follows immediately that indeed $\mathcal{Z}_{U,\{ 0 \}} = \Z_p$ in this case.
\end{proof}

Now the proof of \cref{pro:orbits_of_stabilizer_pq} is almost immediate:

\begin{proof}[Proof of \cref{pro:orbits_of_stabilizer_pq}]
    Consider the cycle set $X = \Z_p \times_{\Gamma} \Z_q$ where $\Gamma$ is nonconstant. Then $H = \left\langle \Gamma_x \right\rangle \leq \Z_q^{\Z_p} \cong \mathbb{F}_q^{\Z_p}$, where $\Gamma_x(y) = \Gamma^0(y-x)$, can be considered a $\Z_p$-invariant subspace of $\mathbb{F}_q^{\Z_p}$. If $\Gamma_0 = \alpha \cdot \chi$ for some character $\chi: \Z_p \to \Z_q^{\ast}$ and $\alpha \neq 0$, then $\Gamma_x = \chi(-x) \cdot \Gamma^0$, which implies $\dim_{\F_q}H = 1$. By \cref{lem:zeros_of_invariant_subspaces}, $\mathcal{Z}_{H,\{ 0\}} = \Z_p$ which implies by \cref{lem:zeros_of_stabilizer} that the orbits of $\Pi_1(X,x_0)$ on $X$ are given by all singletons $\{ (x,a) \}$.

    If $\Gamma^0$ is \emph{not} proportional to a character, then $\dim_{\mathbb{F}_q}H \geq 2$ and \cref{lem:zeros_of_invariant_subspaces} implies that $\mathcal{Z}_{H,\{ 0\}} = \{ 0 \}$. By \cref{lem:zeros_of_stabilizer}, the orbits of $\Pi_1(X,x_0)$ on $X$ are the singletons $\{ (0,a) \}$ ($a \in \Z_q$) and the sets $\{ x\} \times \Z_q$ ($x \in \Z_p^{\ast}$).
\end{proof}

We are now in the position to provide a description of the indecomposable cycle sets $X$ with $|X|= pqr$ and $\mpl(X)=3$. Recall that we assume $|X^{(1)}|=pq$ and $|X^{(2)}|=p$. By \cref{thm:classification_pq}, we can write $X^{(1)} = \Z_p \times_{\Gamma} \Z_q$ for some nonconstant $\Gamma: \Z_p \times \Z_p \to \Z_q$.

We begin with the case where $\Gamma^0 = \Gamma(0,-)$ is proportional to a character. \cref{pro:multiples_of_phi_give_isomorphic_extensions} shows that scalings of $\Gamma^0$ do not affect the isomorphism type $X^{(1)}$, so we may assume that $\Gamma^0$ is itself a nontrivial character, i.e. $\Gamma^0(x) = \xi^x$ for some $\xi \in \Z_q^{\ast}$ with $\xi^p = 1$. The extension of $\Gamma^0$ to a $\G(X^{(1)})$-invariant map is clearly given by $\Gamma(x,y) = \xi^{y-x}$.

\begin{thm} \label{thm:pqr_with_uniconnected_retraction}
    Suppose that $X$ is an indecomposable cycle set with $|X|=pqr$ and let $X^{(1)} = \Z_p \times_{\Gamma} \Z_q$ with $\Gamma(x,y) = \xi^{y-x}$ for some $\xi \in \Z_q^{\ast}$ such that $\xi^p = 1$. Then $X \cong \Z_p \times \Z_q \times \Z_r$ with the binary operation
    \begin{equation} \label{eq:mp3_with_uniconnected_retraction}
    (x,a,s) \ast (y,b,t) = (y+1,b+\xi^{y-x}, t + \Phi^0(y-x, b - \xi^{y-x}a))
    \end{equation}
    for some nonconstant $\Phi^0: \Z_p \times \Z_q \to \Z_r$. On the other hand, each nonconstant map $\Phi^0: \Z_p \times \Z_q \to \Z_r$ defines an indecomposable cycle set by means of \cref{eq:mp3_with_uniconnected_retraction}.
\end{thm}

Note that we do not discuss the issue when the cycle set described by \cref{eq:mp3_with_uniconnected_retraction} is exactly of multipermutation level $3$. This comment applies to all subsequent classifications, as well!

\begin{proof}
    By \cref{pro:orbits_of_stabilizer_pq}, the cycle set $X^{(1)}$ is uniconnected. It is clear that each element of $\G(X^{(1)})$ is of the form 
    \[
    \pi_{(d,k)}: \Z_p \times \Z_q \to \Z_p \times \Z_q ; \quad (x,a) \mapsto (x+d,a + k \cdot \xi^x)
    \]
    for $d \in \Z_p, k \in \Z_q$. As $|\G(X^{(1)})| = pq$, we have $\pi_{(d,k)} \in \G(X^{(1)})$ for all $d \in \Z_p$, $k \in \Z_q$.

    By \cref{cor:coprime_extensions_of_a_uniconnected_cycle_set}, there is an isomorphism $X \cong X^{(1)} \times_{\Phi} \Z_r$. Using \cref{eq:parallel_extensions_expressed_by_gamma0}, we use the map $\Phi^0 = \Phi((0,0),-): X^{(1)} \to \Z_r$ to express the multiplication on $X^{(1)} \times_{\Phi} \Z_r$ by
    \begin{equation*}
    ((x,a),s) \ast ((y,b),t) = ((x,a) \ast (y,b), t + \Phi^0(\lambda_g^{-1}(y,b))) \textnormal{ where } \lambda_g(0,0) = (x,a). \tag{$\ast$}
    \end{equation*}
    Given $(x,a),(y,b) \in \Z_p \times \Z_q$, we see that $\pi_{(x,a)}(0,0)=(x,a)$. There is a $g \in \G(X^{(1)})$ such that $\lambda_g = \pi_{(x,a)}$. This $g$ satisfies $\lambda_g^{-1}(y,b) = (y-x,b- \xi^{y-x}a)$. Plugging this into $(\ast)$ results in \cref{eq:mp3_with_uniconnected_retraction}.

    It is obvious that the construction results in a cycle set as \cref{eq:mp3_with_uniconnected_retraction} expresses the operation in a parallel extension in terms of $\Phi^0$. This cycle set is seen to be indecomposable by an application of \cref{cor:indecomposability_criterion}.
\end{proof}

We now consider the case when $X^{(1)}$ is \emph{not} uniconnected, i.e. $X^{(1)} = \Z_p \times_{\Gamma} \Z_q$ where the map $\Gamma^0: \Z_p \to \Z_q$ is \emph{not} proportional to a character.

To this end, we begin by constructing the $X^{(1)}$-graded $\G(X^{(1)})$-modules $A$ with $|A_x| = r$, that is $A_x \cong \Z_r$ for all $x \in X^{(1)}$. The experienced reader may have noticed that if a set $Y$ is acted upon transitively by a group $G$, then each $Y$-graded $G$-module $A = \bigoplus_{x \in Y}A_x$ is isomorphic to the induced module $\Ind_{G_0}^GA_{x_0}$ where $G_0$ is the stabilizer of some element $x_0 \in Y$, as the $A_x$ form a \emph{system of imprimitivity} in the sense of \cite[§50.1]{CR_Repr}. But even without this knowledge, it is not hard to see that $A$ is already determined up to isomorphism by the $G_0$-module $A_{x_0}$.

Considering $X^{(1)} = \Z_p \times_{\Gamma} \Z_q$, we want to find the $X^{(1)}$-graded $\G(X^{(1)})$-modules $A$ with $|A_x| = \Z_r$ for all $x \in X^{(1)}$. We fix $(0,0) \in X^{(1)}$ as our base point. Let $K = \ker(\G(X^{(1)}) \to \G(X^{(2)}))$. It turns out that the stabilizer
\[
\G(X^{(1)})_0 = \Pi_1(X^{(1)},(0,0)) \leq K.
\]
This can be seen as follows: if $\Bar{g}$ denotes the image of $g \in \G(X^{(1)})_0$ under the homomorphism $\G(\ret): \G(X^{(1)}) \to \G(X^{(2)})$, we observe that $\lambda_{\Bar{g}}(0) = 0$ which implies $\Bar{g} = 0$ as $X^{(2)}$ is uniconnected.

We use \cref{cor:generating_kernel_of_parallel_extension} to describe $K$ as the trivial brace over the abelian group
\[
K = \left\langle \Gamma_x : x \in \Z_p \right\rangle \leq \Z_q^{\Z_p}
\]
where $\Gamma_x: \Z_p \to \Z_q$ is given by $\Gamma_x(y) = \Gamma(x,y) = \Gamma^0(y-x)$. We therefore identify $K$ - and $\G(X^{(1)})_0$, in particular - with a subgroup of $\Z_q^{\Z_p}$ and write its elements by functions $f: \Z_q \to \Z_p$ that act on $X^{(1)}$ as $\lambda_f(x,a) = (x,a+f(x))$. It is now easy to determine the $\G(X^{(1)})_0$-module structures on $\Z_r$ - each such structure is given by a character $\Bar{\chi}: \G(X^{(1)})_0 \to \Z_r^{\ast}$, and as the group $\Z_q^{\Z_p}$ and its subgroup $K$ are elementary abelian, $\G(X^{(1)})_0$ is a summand thereof, therefore each character on $\G(X^{(1)})_0$ is the restriction of a character on $K$ resp. $\Z_q^{\Z_p}$. Identifying $\G(X^{(1)})_0$ with an (additive) subgroup of $\Z_q^{\Z_p}$, it follows that we can write
\[
\Bar{\chi} = \chi|_{\G(X^{(1)})_0}, \textnormal{where } \chi: K \to \Z_r ; \quad f \mapsto \prod_{x \in \Z_p^{\ast}} \xi_x^{f(x)}.
\]
Where the elements $\xi_x \in \Z_r$ ($x \in \Z_p$) satisfy $\xi_x^q = 1$. Note that we exclude $x = 0$ because $f(0)=0$, due to $f$ stabilizing $(0,0)$.

Therefore, $\G(X^{(1)})_0$ acts on $\Z_r$ via $f \cdot z = z \cdot \prod_{x \in \Z_p^{\ast}} \xi_x^{f(x)}$. The unique extension to an $X^{(1)}$-graded $\G(X^{(1)})$-group $A$ with $A_{(0,0)} = \Z_r$ is given by $A_{(x,a)} = \Z_r$ for all $(x,a) \in X^{(1)}$, where $\G(X^{(1)})$ acts by

\begin{align}
    \rho_{(d,f)}: A_{(x,a)} & \to A_{(x+d,a + f(x))} \\
    z & \mapsto z \cdot \prod_{y \in \Z_p^{\ast}} \xi_y^{f(y+x)},  \label{eq:graded_X_modules}.
\end{align}
Here, $(d,f)$ denotes the element in $\G(X^{(1)}) \leq \Z_p \ltimes \Z_q^{\Z_p}$ that maps $(x,a) \mapsto (x+d,a+f(x))$.

Note that given the character $\Bar{\chi}$ on $\G(X^{(1)})_0$, the extension to an $X$-graded $\G(X)$-module is unique up to isomorphism as $\G(X^{(1)})$ acts transitively on $X^{(1)}$, so we only have to check that this is indeed an $X^{(1)}$-graded $\G(X^{(1)})$-module: we consider the action of $\rho_{(e,g)}\rho_{(d,f)}$ on an element $z \in A_{(x,a)}$:

\begin{allowdisplaybreaks}
\begin{align*}
    (\rho_{(e,g)}\rho_{(d,f)} )(z) & = \rho_{(e,g)}\left( z \cdot \prod_{y \in \Z_p^{\ast}} \xi_y^{f(y+x)} \right) \\
    & = z \cdot \prod_{y \in \Z_p^{\ast}} \xi_y^{g(y+x+d)} \cdot \prod_{y \in \Z_p^{\ast}} \xi_y^{f(y+x)} \\
    & = z \cdot \prod_{y \in \Z_p^{\ast}} \xi_y^{g(y+x+d)+f(y+x)} \\
    & = \rho_{(e,g)(d,f)}(z).
\end{align*}
\end{allowdisplaybreaks}

It will occasionally be practical to take shifts into account. For that sake, we define for $x \in \Z_p$ the characters
\begin{equation} \label{eq:def_of_chi_x}
\chi_x: K \to \Z_r; \quad f \mapsto \prod_{y \in \Z_p^{\ast}} \xi_y^{f(y+x)}.
\end{equation}
Note that $\chi = \chi_0$. With this definition \cref{eq:graded_X_modules} can be rewritten for $z \in A_{(x,a)}$ as
\begin{equation} \label{eq:action_of_chi_x}
    \rho_{(d,f)}(z) = \chi_x(f) \cdot z \in A_{(x+d,a+f(x))}.
\end{equation}

We fix the $X^{(1)}$-graded $\G(X^{(1)})$-module $A = \Z_r^{X^{(1)}}$ where $\G(X^{(1)})$ acts as in \cref{eq:graded_X_modules}. We have to find conditions for a function $\Phi^0: \Z_p \times \Z_q \to A$ to satisfy the $\G(X^{(1)})_0$-equivariance given by \cref{eq:equivariance_for_pi1}. Rewritten in terms of the module structure given by \cref{eq:action_of_chi_x}, this translates to
\begin{equation} \label{eq:equivariance_over_zpxzq}
    \Phi^0(x,a+f(x)) = \Phi^0(x,a) \cdot \chi_x(f)
\end{equation}
for $f \in \G(X^{(1)})_0$. Note that $f \in \G(X^{(1)})_0$ implies $f(0)=0$. Furthermore, we observe that the existence of an $f \in \G(X^{(1)})_0$ with $\chi_x(f) \neq 1$ and $f(x)=0$ forces $\Phi^0(x,a) = 0$. On the other hand, suppose that $x \in \Z_p^{\ast}$. If $\chi_x(f) = 1$ for all $f \in \G(X^{(1)})_0$ with $f(x)=0$, then given the value $\Phi(x,0)$, the values $\Phi^0(x,a)$ can easily be reconstructed as
\begin{equation}
\Phi^0(x,a) = \Phi^0(x,0) \cdot \chi_x(f)
\end{equation}
where $f \in \G(X^{(1)})_0$ is chosen such that $f(x) = a$. Observe that such an $f$ exists because of \cref{pro:orbits_of_stabilizer_pq} and the assumption that $x \neq 0$, and independence from the choice of $f$ follows from the assumption that $\chi_x(f) = 1$ for all $f \in \G(X^{(1)})_0$ with $f(x)=0$. We define the set
\begin{equation} \label{eq:definition_n_chi}
\mathcal{N}_{\chi} = \{ x \in \Z_p: \forall f \in \G(X^{(1)})_0: f(x)=0 \Rightarrow \chi_x(f) = 1 \}.
\end{equation}

\medskip

We now investigate two cases:

\noindent\textbf{Case 1: $0 \in \mathcal{N}_{\chi}$.}

This implies that $\Bar{\chi} = \chi|_{\G(X^{(1)})_0}$ is the trivial character. Therefore, the module $A$ is isomorphic to the permutation module $\Z_r^{X^{(1)}}$. In this case, we have a parallel extension $X^{(1)} \times_{\Phi} \Z_r \cong X$.

We reconstruct $\Phi$ from the restriction $\Phi^0 = \Phi((0,0),-): X^{(1)} \to \Z_r$ which is $\Pi(X^{(1)},(0,0))$-invariant in the sense of \cref{eq:equivariance_for_parallel_extensions}. By \cref{pro:orbits_of_stabilizer_pq}, this is the same as saying that $\Phi^0$ is constant on $\{ x \} \times \Z_q$ for $x \in \Z_p^{\ast}$. Therefore, there are functions $\Phi^0_1: \Z_q \to \Z_r$ and $\Phi^0_2: \Z_p^{\ast} \to \Z_r$ such that
\[
\Phi^0(x,a) = \begin{cases}
    \Phi^0_1(a) & x = 0 \\
    \Phi^0_2(x) & x \neq 0.
\end{cases}
\]
It is quickly checked that the extension of $\Phi^0$ to a $\G(X^{(1)})$-invariant map $\Phi: X^{(1)} \times X^{(1)} \to \Z_r$ is given by
\[
\Phi((x,a),(y,b)) = \begin{cases}
    \Phi^0_1(b-a) & x = y \\
    \Phi^0_2(y-x) & x \neq y.
\end{cases}
\]
Plugging everything into \cref{eq:cycle_operation_parallel_extension}, the resulting cycle set is then described as $X = \Z_p \times \Z_q \times \Z_r$ with the operation
\[
(x,a,s) \ast (y,b,t) = \begin{cases}
    (y+1,b+\Gamma^0(y-x), t+ \Phi^0_1(b-a)) & x = y, \\
    (y+1,b+\Gamma^0(y-x), t+ \Phi^0_2(x-y)) & x \neq y.
\end{cases}
\]
By \cref{cor:indecomposability_criterion}, $X$ is indecomposable if and only if $\Phi^0_1$ or $\Phi^0_2$ is nonconstant.

\begin{rem}
    Assume that we are given another prime $s \not\in \{p,q,r \}$, then there is an obvious way to extend this construction: let $\Omega_1^0: \Z_r \to \Z_s$, $\Omega_2^0: \Z_q^{\ast} \to \Z_s$, $\Omega_3^0: \Z_p^{\ast} \to \Z_s$, then the cycle set structure constructed on $X = \Z_p \times \Z_q \times \Z_r$ can be extended to give one on $X \times \Z_s$ by means of
    \[
    (x,a,s,v) \ast (y,b,t,w) = \begin{cases}
    (y+1,b+\Gamma^0(y-x), t+ \Phi^0_1(b-a),w +\Omega_1^0(t-s)) & x = y;\ a = b \\
    (y+1,b+\Gamma^0(y-x), t+ \Phi^0_1(b-a),w + \Omega_2^0(b-a)) & x = y; a \neq b, \\
    (y+1,b+\Gamma^0(y-x), t+ \Phi^0_2(y-x),w + \Omega_3^0(y-x)) & x \neq y.
\end{cases}
    \]
    It is obvious how to iterate this construction.
\end{rem}

\medskip

\noindent\textbf{Case 2: $0 \not\in \mathcal{N}_{\chi}$.}

For $x \in \mathcal{N}_{\chi}$, we now define a map $\phi_x^0: \Z_p \times \Z_q \to A$ by
\begin{align*}
    \phi_x^0(y,a) & \in A_{y,a}, \\
    \textnormal{for } f \in \G(X^{(1)})_0: \quad \phi_x^0(y,f(x)) & = \begin{cases}
        0 & x \neq y, \\
        \chi_x(f) & x = y.
    \end{cases}
\end{align*}
It is easy to see that $\phi_x^0$ is well-defined and that $\phi_x^0(y,a+f(y)) = f \cdot \phi_x^0(y,a)$ for all $f \in \G(X^{(1)})_0$, $(y,a) \in \Z_p \times \Z_q$, and any $\G(X^{(1)})_0$-equivariant $\Phi^0: \Z_p \times \Z_q \to A$ can be assembled as
\begin{equation} \label{eq:decompose_phi0}
\Phi^0 = \sum_{x \in \mathcal{N}_{\chi}} c_x \phi_x^0
\end{equation}
with $c_x \in \Z_r$ ($x \in \mathcal{N}_{\chi}$). Note that it is necessary and sufficient for $X \otimes_{\Phi}A$ to be indecomposable that $c_x \neq 0$ for some $x \in \mathcal{N}_{\chi}$.

We are left with expressing the operation on $X$ in terms of $\Phi^0$. We make the assumption that there is an $x_0 \in \Z_p$ with $\Gamma_{x_0}(0) = 1$\footnote{By \cref{pro:multiples_of_phi_give_isomorphic_extensions}, this assumption does not pose any restriction!}. Recall that $\Gamma_x(y) = \Gamma(x,y) = \Gamma^0(y-x)$.

We now define for $d \in \Z_p$, $k \in \Z_q$ the permutations
    \[
    \pi_{(d,k)}: X^{(1)} \to X^{(1)} ; \quad (x,a) \mapsto (x+d, a + k \cdot \Gamma_{x_0}(x)).
    \]
    By \cref{cor:generating_kernel_of_parallel_extension}, $\pi_{(d,k)} \in \G(X^{(1)})$. We now identify $A^{\sca}$ with $X^{(1)} \times \Z_r$ and, using \cref{eq:twisted_extensions_expressed_by_phi0}, we express the operation on $X = X^{(1)} \times \Z_r$ as
    \begin{equation}
        ((x,a),s) \ast ((y,b),t) = ((x,a)\ast(y,b), \underbrace{\lambda_{(x,a)}^{-1} \cdot (t + g \cdot \Phi^0(\lambda_g^{-1}(y,b)))}_{=: F(x,a,y,b,t)}) \quad \textnormal{ where } \lambda_g(0,0) = (x,a)   \tag{$\ast$}
    \end{equation}
    It turns out that $\pi_{(x,a)}(0,0) = (x,a)$, so let $g \in \G(X^{(1)})$ be such that $\lambda_g = \pi_{(x,a)}$. With this $g$, we furthermore get $\lambda_g^{-1}(y,b) = (x-y,b - a \cdot \Gamma_{x_0}(y-x))$.
    Note that $\lambda_{(x,a)}^{-1}(y,b) = (y+1,b + \Gamma_x(y))$. Using the description \cref{eq:graded_X_modules} for the module structure, we can therefore calculate
    \begin{align*}
        F(x,a,y,b,t) & = \lambda_{(x,a)}^{-1} \cdot \underbrace{(t + g \cdot \overbrace{\Phi^0(\lambda_g^{-1}(y,b))}^{\in A_{(x-y,b-a \cdot \Gamma^0(y-x-x_0))}} )}_{\in A_{(y,b)}} \\
        & = \chi_y(\Gamma_x) \cdot \left( t + \chi_{x-y}(a \cdot \Gamma_{x_0}) \cdot \Phi^0(x-y,b - a \cdot \Gamma_{x_0}(y-x)) \right).
    \end{align*}

We can now summarize our findings! In the following we write, as usual, $K = \genrel{\Gamma_x}{x \in \Z_p}$. Furthermore, we identify $\G(X^{(1)})_0$ with $K_0 = \{f \in K: f(0)=0 \}$.

\begin{thm} \label{thm:pqr_with_non_uniconnected_retraction}
 Suppose $X$ is an indecomposable cycle set with $\mpl(X)=3$, $|X|=pqr$ and $X^{(1)} = \Z_p \times_{\Gamma} \Z_q$, where $\Gamma^0$ is nonzero but \emph{not} proportional to a character. Then
 \begin{enumerate}
     \item either $X \cong \Z_p \times \Z_q \times \Z_r$, with the operation
     \[
     (x,a,s) \ast (y,b,t) = \begin{cases}
    (y+1,b+\Gamma^0(y-x), t+ \Phi^0_1(b-a)) & x = y, \\
    (y+1,b+\Gamma^0(y-x), t+ \Phi^0_2(x-y)) & x \neq y,
\end{cases}
     \]
     where $\Phi_1^0: \Z_q \to \Z_r$ and $\Phi^0_2: \Z_p^{\ast} \to \Z_r$ are maps that are not both identically zero.
     \item or $X \cong \Z_p \times \Z_q \times \Z_r$, with the operation
     \[
     (x,a,s) \ast (y,b,t) = \left(y+1,b + \Gamma^0(y-x),  \chi_y(\Gamma_x) \cdot \left( t + \chi_{x-y}(a \cdot \Gamma_{x_0}) \cdot \Phi^0(x-y,b - a \cdot \Gamma_{x_0}(y-x)) \right) \right).
     \]
     Here, $\Gamma^0(-x_0)=1$ and $\chi: K \to \Z_r^{\ast}$ is a character whose restriction to $K_0$ is not identically $1$, and with
     \[
     \mathcal{N}_{\chi} = \{ x \in \Z_p: \forall f \in K_0: f(x)=0 \Rightarrow \chi_x(f) = 1 \},
     \]
     we have $\Phi^0 = \sum_{x \in \mathcal{N}_{\chi}}c_x \phi_x^0$, with $c_x \in \Z_r$, not all $0$, where
\begin{align*}
    \phi_x^0: \Z_p \times \Z_q & \to \Z_r, \\
    \phi_x^0(y,a) & = \begin{cases}
        0 & x \neq y, \\
        \chi_x(f) & x = y, \textnormal{where } f \in K_0 \textnormal{ satisfies } f(x)=a.
    \end{cases}
\end{align*}
 \end{enumerate}
 In both cases, the conditions for $\Phi_1^0,\Phi_2^0$ resp. $\Phi^0$ are sufficient to define an indecomposable cycle set (that might be of multipermutation level $< 3$).
\end{thm}

Note that the second case of the theorem can only occur when $q \mid r-1$.

The theorem suggests that it is very unlikely to achieve a general explicit description of cycle sets of squarefree order whose size has more than three prime factors.

\section{A comment on Lebed-Vendramin cohomology}

We recapitulate the concept of an \emph{abelian extension} of cycle sets that has been investigated by Lebed, Vendramin \cite{Lebed_Vendramin_Homology}:

\begin{defn}[{\cite[Definition 9.9]{Lebed_Vendramin_Homology}}]
    Let $p: Y \twoheadrightarrow X$ be a (constant) extension of cycle sets and let $B$ be an abelian group with an action $B \times Y \to Y; (a,y) \mapsto y + a$ that restricts to regular actions on the fibers $p^{-1}(x)$ ($x \in X$). Then the couple $(p:Y \to X, B)$ is called an \emph{abelian extension of $X$ by $B$} if
    \[
    y \ast (z + a) = (y \ast z) + a
    \]
    holds for all $y,z \in Y$, $a \in B$ (note that the condition $(y+a)\ast z = y \ast z$ in the original definition is contained in the assumption that $p$ is a constant extension).
\end{defn}

They prove that each extension of a cycle set $X$ by $B$ can be parametrized by a \emph{$2$-cocycle} $\Phi: X \times X \to B$ that satisfies
\[
\Phi(x,z) + \Phi(x \ast y,x \ast z) =  \Phi(y,z) + \Phi(y \ast x, y \ast z)
\]
for all $x,y,z \in X$. Note that parallel extensions are an instance of this construction!

Given a $2$-cocycle $\Phi$, one can recover, up to equivalence of extensions, $Y$ as $X \times B$ with the operation
\[
(x,a) \ast (y,b) = (x \ast y, b + \Phi(x,y)),
\]
together with the projection $Y \twoheadrightarrow X$; $(x,a) \mapsto x$. Denote this cycle set as $X \times_{\Phi} B$. Furthermore, $2$-cocycles $\Phi,\Gamma: X \times X \to B$ define equivalent extensions\footnote{Here, an equivalence of extensions furthermore respects the actions of $B$} if and only if there is a map $\gamma: X \to B$ such that
\begin{equation} \label{eq:lebed_vendramin_cohomologous}
\Phi(x,y) - \Gamma(x,y) = \gamma(x \ast y) - \gamma(y).
\end{equation}

In view of the proof of \cref{thm:coprime_extensions_are_twisted_extensions}, it is natural to generalize the concept as follows:

\begin{defn}
    Let $X$ be a cycle set and $A$ an $X$-graded $\G(X)$-module. Then a (constant) extension $p: Y \twoheadrightarrow X$ is called an \emph{abelian extension of $X$ by $A$} if there is an action of $A$ on $Y$ given by $(a,y) \mapsto y+a$ such that:
    \begin{enumerate}
        \item for $x \in X$, the component $A_x \leq A$ acts regularly on the fiber $p^{-1}(x)$,
        \item for $x,y \in X$ with $y \neq x$, the component $A_y \leq A$ has trivial action on $p^{-1}(x)$,
        \item for $x \in X$, $y \in Y$, $a \in A$, the action satisfies
        \[
        x \ast (y + a) = x \ast y + \lambda_x^{-1} \cdot a.
        \]
    \end{enumerate}
\end{defn}

Paralleling the proofs of Lebed and Vendramin, one establishes that an extension of $X$ by $A$ can be represented as $X \otimes_{\Phi} A = A^{\sca}$, together with a binary operation
\begin{equation} \label{eq:general_A_extension}
(x,a) \ast (y,b) = (x \ast y, \lambda_x^{-1} \cdot (b + \Phi(x,y)))
\end{equation}
for some map $\Phi: X \times X \to A$ with $\Phi(x,y) \in A_y$ ($x,y \in X$). \cref{eq:general_A_extension} defines a cycle set if and only if $\Phi$ is a \emph{twisted} $2$-cocycle, meaning that
\begin{equation} \label{eq:twisted_cocycles}
\Phi(x,z) + \lambda_x \cdot \Phi(x \ast y, x \ast z) = \Phi(y,z) + \lambda_y \cdot \Phi(y \ast x, y \ast z).
\end{equation}
Note that if equivariance in form of \cref{eq:phi_equivariance} is satisfied, then $\Phi(x \ast y, x \ast z) = \lambda_x^{-1} \cdot \Phi(y,z)$, and $\Phi(y \ast x, y \ast z) = \lambda_y^{-1} \cdot \Phi(x,z)$, so \cref{eq:twisted_cocycles} is indeed satisfied. Furthermore, $\Phi,\Gamma$ define equivalent extensions by $A$, if there is a $\gamma: X \to A$ with $\gamma(x) \in A_x$ ($x \in X$) such that
\[
\Phi(x,y) - \Gamma(x,y) = c_y - \lambda_x \cdot c_{x \ast y}.
\]
Considering the proof of \cref{thm:coprime_extensions_are_twisted_extensions}, it becomes apparent that an extension $p: Y \to X$ of indecomposable cycle sets can be parametrized as an extension by an $X$-graded $\G(X)$-module if $K = \ker(\G(Y) \twoheadrightarrow \G(X))$ induces regular actions on the fibers $p^{-1}(x)$ ($x \in X$). This condition is satisfied in the case of a coprime extension.

What we have essentially shown is that for an $X$-graded $\G(X)$-module $A$ with $\gcd(|\G(X)|,|A|) = 1$, each twisted $2$-cocycle $\Phi: X \times X \to A$ is cohomologous to a unique equivariant $2$-cocycle. This observation implies the \emph{need} for a development of \emph{twisted} Lebed-Vendramin cohomology in order to get more conceptual proofs of the results in this article.

\bibliography{references}
\bibliographystyle{abbrv}

\end{document}